\documentclass[11pt,oneside, %draft
]{amsart}
\usepackage[T3,T1]{fontenc}
\usepackage[english]{babel}
\usepackage{fouriernc} % Fourier fonts instead of Computer Modern

\usepackage[shortlabels]{enumitem}
\usepackage{comment}
\usepackage{adjustbox}
\usepackage{caption}
\usepackage{subcaption}
\usepackage{quiver}

\usepackage{geometry} % page geometry
\usepackage{xcolor} % colors
\usepackage{tikz} % drawing

\usetikzlibrary{cd}
\usetikzlibrary{calc, shapes}

\usepackage{graphicx, graphics}    

\newcommand{\remnant}{twisted endomorphism}%{remnant}
\newcommand{\coremant}{twisted coendomorphism}%{coremnant}
\newcommand{\boxmaps}[2]{{#1}\diamondsuit{#2}}

\usepackage{tikz}
\usetikzlibrary{matrix}
\usetikzlibrary{arrows,backgrounds,patterns,scopes,external,
    decorations.pathreplacing,
    decorations.pathmorphing
}

\usepackage{hyperref} % hyperlinks
\usepackage{stmaryrd}
\usepackage{fancyvrb} %for verbatium
\usepackage{dirtytalk}
\usepackage{pdflscape}
\usepackage{amsmath,amsthm,amssymb}
\usepackage[all,cmtip,2cell]{xy}
\UseTwocells

\usepackage{soul}
\sethlcolor{cyan}

\newcommand{\sh}[1]{{\ensuremath{\hspace{1mm}\makebox[-1mm]{$\langle$}\makebox[0mm]{$\langle$}\hspace{1mm}{#1}\makebox[1mm]{$\rangle$}\makebox[0mm]{$\rangle$}}}}

\usepackage[capitalise]{cleveref}  % load after hyperref
\newcommand{\clevertheorem}[3]{%
  \newtheorem{#1}[thm]{#2}
  \crefname{#1}{#2}{#3}
}

\numberwithin{equation}{section} %% Comment out for sequentially-numbered
\numberwithin{figure}{section} %% Comment out for sequentially-numbered

\usepackage{thmtools, thm-restate}

\theoremstyle{plain} % bold environment name, italic text
\newtheorem{thm}{Theorem}[section]
\crefname{thm}{Theorem}{Theorems}
\newtheorem*{thm*}{Theorem}
\clevertheorem{prop}{Proposition}{Propositions}
\newtheorem*{prop*}{Proposition}
\clevertheorem{lem}{Lemma}{Lemmas}
\crefname{lm}{Lemma}{Lemmas}
\clevertheorem{cor}{Corollary}{Corollaries}
\clevertheorem{conj}{Conjecture}{Conjectures}
\theoremstyle{definition} % bold environment name, plain text
\clevertheorem{defn}{Definition}{Definitions}
\newtheorem*{defn*}{Definition}
\clevertheorem{notn}{Notation}{Notations}
\clevertheorem{ex}{Example}{Examples}

\theoremstyle{remark} % italic environment name, plain text
\clevertheorem{rmk}{Remark}{Remarks}
\newtheorem*{rmk*}{Remark}

\makeatletter\let\c@equation\c@thm\makeatother
\makeatletter\let\c@figure\c@thm\makeatother

\crefname{figure}{Figure}{Figures}
\crefname{equation}{Display}{Displays} % give 'equation' a more general name 
\crefname{eq}{Display}{Displays}
\crefname{eqn}{Display}{Displays}

\newcommand{\cc}{\mathcal{C}}

\newcommand{\B}{\mathcal{B}}

\newcommand{\T}{\mathcal{T}}

\newcommand{\cb}{\overline{\Lambda}\mathcal{B}}

\newcommand{\ltwo}[1]{#1}%{\breve{#1}}
\newcommand{\lone}[1]{#1}%{\invbreve{#1}}
\newcommand{\cob}[1]{\overline{\Lambda}{#1}}

\DeclareMathOperator{\Hom}{Hom}

\DeclareMathOperator{\Mod}{\mathcal{M}\!\emph{od}}
\DeclareMathOperator{\Ring}{\mathcal{R}\!\emph{ing}}
\DeclareMathOperator{\Vect}{Vect}

\DeclareMathOperator{\tr}{tr}
\DeclareMathOperator{\Tr}{\mathbb{T}r}
\DeclareMathOperator{\co}{cotr}

\DeclareMathOperator{\id}{id}

\usepackage{xargs}  
\usepackage{todonotes}
\newcommandx{\kate}[3][1=]{\todo[inline,linecolor=green,backgroundcolor=green!10,bordercolor=green,#1]{Kate ({#2}): #3}}
\newcommandx{\katemar}[3][1=]{\todo[linecolor=green,backgroundcolor=green!10,bordercolor=green,#1]{Kate ({#2}): #3}}
\newcommandx{\justin}[2][1=]{\todo[inline,linecolor=blue,backgroundcolor=blue!10,bordercolor=blue,#1]{Justin: #2}}

\newcommandx{\travis}[3][1=]{\todo[inline,linecolor=red,backgroundcolor=red!10,bordercolor=red,#1]{Travis({#2}): #3}}
\newcommandx{\travismar}[3][1=]{\todo[linecolor=red,backgroundcolor=red!10,bordercolor=red,#1]{Travis ({#2}): #3}}
\usepackage{mathtools}
\usepackage{mathrsfs}

\usepackage{tikz-cd} 

\DeclareRobustCommand{\rchi}{{\mathpalette\irchi\relax}}

\usepackage{xargs}  
\usepackage{todonotes}

\usepackage{mathtools}
\usepackage{mathrsfs}

\usepackage{tikz-cd} 

\DeclareRobustCommand{\rchi}{{\mathpalette\irchi\relax}}
\newcommand{\irchi}[2]{\raisebox{\depth}{$#1\chi$}}
\newcommand{\coshh}[1]{\langle \hspace{-0.35ex} \langle #1 \langle \hspace{-0.35ex} \langle}

\DeclareMathOperator{\Aut}{Aut}
\DeclareMathOperator{\Fun}{Fun}

\DeclareMathOperator{\rhov}{\rho}
\DeclareMathOperator{\rhow}{\sigma}
\DeclareMathOperator{\Mat}{Mat}

\DeclareSymbolFont{tipa}{T3}{cmr}{m}{n}
\DeclareMathAccent{\invbreve}{\mathalpha}{tipa}{16}
\newcommand{\xto}{\xrightarrow}

\theoremstyle{plain}

\newcommand{\repdual}{shadow dualizable}

\hypersetup{
 pdfkeywords={latex starter,template},
 pdfauthor={Niles Johnson},
}

\title{Character theory for 2-representations}

\author{Kate Ponto}
\address{University of Kentucky }
\email{kate.ponto@uky.edu}
\author{Travis Wheeler}
\address{Kettering University}
\email{twheeler@kettering.edu}

\date{\today}

\begin{document}

\maketitle
\begin{abstract}
Group representations very comfortably admit a higher categorical generalization.  This includes induction and restriction for 2-representations that very directly parallel classical definitions.  An appropriate character theory for 2-representations has been more opaque. 

In this paper we draw from a range of motivations to provide an intuitive character theory for 2-representations.  It is built from the bicategorical trace developed to describe topological fixed point theory.  We use the secondary traces of Ben-Zvi and  Nadler which connect to  Riemann-Roch and Lefschetz-type theorems in the context of differential graded categories and Fourier-Mukai transforms.  Additional sources of motivation include  Lipmans' work on Grothendieck duality and, crucially, the Bartlett and Ganter--Kapranov approaches to characters.  (The second of these was in turn motivated by Hopkins, Kuhn and Ravenel's work on equivariant
Morava $E$-theories.)
Our characters are compatible with restriction and induction and, in cases that are sufficiently finite, admit a secondary character.  

\end{abstract}
\setcounter{tocdepth}{1}
\makeatletter
\def\l@subsection{\@tocline{2}{0pt}{2.5pc}{5pc}{}}
\makeatother

\tableofcontents

\section{Introduction}

While not the most common definition, a linear representation of a group $G$ is a functor 
\[BG\to \Vect_k\]
where $BG$ is the category with a single object and the group $G$ as morphisms and $\Vect_k$ is the category of vector spaces and linear maps over a field $k$.  The single object of $BG$ picks out a vector space and the elements of the group $G$ define endomorphisms of that vector space. Motivated by this perspective, if we make $\Vect_k(k,k)$ a category by giving it only identity morphisms, the character of a representation is a functor 
\begin{equation}\label{eq:1_character}\Lambda BG\to \Vect_k(k,k)\end{equation}
where $\Lambda BG$ has objects the elements of $G$ and a morphism $g_1\to g_2$ for each $h\in G$ such that $g_1h=hg_2$.  The functor in \eqref{eq:1_character} can be identified with the usual definition of the character by noting that a transformation $k\to k$ is the same as an element of $k$ and, since the target of the functor has only identity morphisms, conjugate elements of $G$ have the same image in $\Vect_k(k,k)$.

These are  elaborate definitions for familiar objects, but they have a couple of important benefits.  The first, as in \cite{elgueta,GK}, this definition of a representation admits an easy generalization to 2-categories: 
{\bf 2-representation} of $G$ in a bicategory $\mathcal{B}$
is a strong 2-functor 
\[BG \to \mathcal{B}\] where $BG$ is as above with the addition of identity 2-cells. 
Second,  in this perspective both a representation and its character are the same kind of object -- they are functors.  This allows us to treat them on more equal footing.  

In this paper we embrace this perspective and combine it with the theory of traces in bicategories \cite{p:thesis}.  The following is a specialization of a main result. 
\begin{thm*}
Let $\rho\colon BG\to \mathcal{B}$ be a 2-representation in a bicategory $\mathcal{B}$.  If $\mathcal{B}$ has a shadow 
(\cref{defn:shadow}), there is a functor $\mathcal{X}_\rho$, called the  {\bf character}, 
generalizing the classical character.

Additionally,
\[R_{\Lambda BG}^{\Lambda BH} \mathcal{X}_\rho=\mathcal{X}_{R_{BG}^{BH}} \rho
\quad \quad I_{\Lambda BH}^{\Lambda BG} \mathcal{X}_\rho\xrightarrow{\exists!}\mathcal{X}_{I_{BH}^{BG}} \rho
\quad \quad \mathcal{X}_{C_{BG}^{BH}} \rho \xrightarrow{\exists!}  C_{\Lambda BG}^{\Lambda BH} \mathcal{X}_\rho\]
where $R$, $I$, and $C$ are restriction, induction and coinduction.
\end{thm*}
This is remarkably parallel to the classical results but with additional richness reflecting the higher categorical level.  It is compatible with, but more general than, the approach of \cite{GK}.

Very reassuringly, one very direct illustration of  the greater richness of 2-representations is in their characters.  The character for a 2-representation is  a functor from $\Lambda BG$, but  the target can have non identity morphisms.  The character is no longer conjugation invariant in the strict sense, but  there is a morphisms between the images of  conjugate elements.  In particular, when the character is valued in a category where the (symmetric monoidal) trace \cite{Dold1978-ol} is defined, the character of this character is defined.  As described in \cite{BZN,GK,HKR}, this is very much an expected phenomenon and the inevitability of its evolution in this approach is very satisfying.

\begin{thm*}
Let $\rho\colon BG\to \mathcal{B}$ be a 2-representation in a bicategory $\mathcal{B}$ with a shadow taking values in a symmetric monoidal category. If the images of $\rho(g)$ are sufficiently finite in the shadow target (\cref{def:repdual}), there is a {\bf secondary character} $\mathcal{X}_{\mathcal{X}_\rho}$ whose domain is pairs of commuting elements of $G$ up to simultaneous conjugation.  

Additionally, 
\[R_{\Lambda(\Lambda BG)}^{\Lambda(\Lambda BH)} \mathcal{X}_{\mathcal{X}_\rho}=\mathcal{X}_{\mathcal{X}_{R_{BG}^{BH}} \rho}
\quad \quad I_{\Lambda(\Lambda BH)}^{\Lambda(\Lambda BG)} \mathcal{X}_{\mathcal{X}_\rho}=\mathcal{X}_{\mathcal{X}_{I_{BH}^{BG}} \rho}
\quad \quad  C_{\Lambda(\Lambda BH)}^{\Lambda(\Lambda BG)} \mathcal{X}_{\mathcal{X}_\rho}=\mathcal{X}_{\mathcal{X}_{C_{BH}^{BG}} \rho}\]
for restriction, induction, and coinduction.
\end{thm*}

As mentioned above, this paper is motivated by the greater generality that the trace in bicategories from \cite{p:thesis} can bring to the theory of characters developed by Bartlett in \cite{Bartlett} and  Ganter and Kapranov in \cite{GK}.  As a result, this paper shares motivation with \cite{GK} including the character theory of Hopkins, Kuhn and Ravenel \cite{HKR} defined on tuples of commuting elements.  As with \cite{GK}, our secondary character is defined on pairs of commuting elements and we share their belief that higher categorical analogs are a likely source of characters defined on longer tuples of commuting elements.

In contrast to \cite{GK}, our approach is more flexible about the kind of traces used to define characters. 
The main construction in \cite{GK} is called the {\bf categorical trace} and it is a construction on 1-cells in a bicategory.  Working on an unrelated project \cite{barhite2023bicategorical}, Barhite noticed that the categorical trace is an example of his coshadow.   
Since we define characters using traces \cite{p:thesis} and cotraces \cite{barhite2023bicategorical}, Barhite's observation implies that Ganter and Kapranov's characters are an example of the character defined here.  However, we define characters that are not accessible using the categorical trace.  In particular, we can define characters in terms of Hochschild and topological Hochschild homology.

Another important point of difference between this paper and \cite{GK}, is in the descriptions of  induced representations for the 2-category of categories, functors and natural transformations. We give a more general definition of induced representation, but describe the compatibility between induction and charcters via a unique natural map rather than the equality of the construction in \cite{GK}.

\subsection*{Outline} In \cref{intro:classical} we recall fundamental results in classical representation theory rewritten to foreshadow our generalizations to 2-representations.  \cref{sec:shadows} recalls necessary preliminaries on traces in 2-categories.  In \cref{sec:main} we state the main results of this paper postponing the proofs to \cref{sec:remnant,sec:remnant_functors,sec:traces}. \cref{sec:coshadows} describes how these results apply to the case of coshadows and cotraces.  In the last section, \cref{sec:induction}, we define induction and coinduction for 2-representations and prove compatibility between  character and induction (and coinduction).

\subsection*{Acknowledgments} The authors thank Justin Barhite, Niles Johnson, and  Jordan Sawdy for helpful conversations.

 Both authors were partially supported by NSF grant DMS-2052905.  The first author was also partially supported by NSF grant  DMS-2404503.

 \section{Classical representations}\label{intro:classical}
 To ease the transition to 2-categories we start by recalling (and rephrasing) the relevant parts of classical representation theory so that the 2-categorical analogs will be fairly direct translations (with more elaborate diagrams).  
 
 A linear representation of a group $G$ is a functor 
 \[BG\to \Vect_k\]
 where 
 $BG$ is the category with a single object and $G$ as morphism. 
Our first generalization is the following.
 \begin{defn}\label{defn:classical_rep} Let $G$ be a group and $C$ a category. A $G$-\textbf{representation} in $C$ is a functor \[\rho: BG \to C.\]
\end{defn}

This is too general to define characters.  For that we need a notion of trace for the category $C$ and so we restrict to symmetric monoidal categories and use Dold and Puppe's approach to trace \cite{Dold1978-ol}.

\begin{defn}[Compare to \cref{defn:dual}]\label{defn:dual_early} \cite[1.2]{Dold1978-ol}
An object $X$ in a  monoidal category $(C, \wedge, S)$ is {\bf dualizable} if there is an object $Y$ and morphisms 
$\eta\colon S\to X\wedge Y$ and $\epsilon \colon Y\wedge X\to S$ so that the triangle diagrams commute.
\end{defn}

 In $\Vect_k$ the dualizable objects are the finite dimensional vector spaces.

Classically, the character is a class function from the group to the ground field.  It will be more convenient for us to replace this by a functor whose target category is appropriately trivial.  The first step in this process is the following proposition. 
\begin{prop}[Compare to \cref{Lambda_bicat}]\label{Lambda_bicat_early}
For a  monoidal category $C$, there is a category $\Lambda C$  whose
\begin{itemize}
    \item objects are pairs $(A,f)$ where $A$ is a dualizable object of $C$ and $f$ is an endomorphisms of $A$, and 
    \item a morphism $\psi\colon (A,f)\to (B,g)$ is a isomorphism $\psi\colon A\to B$ in $C$ so that the diagram 
\[\xymatrix{A\ar[r]^f\ar[d]_\psi&A\ar[d]^\psi\\
B\ar[r]^-g&B}\]
commutes.
\end{itemize}  
\end{prop}
  Composition in $C$ defines composition for morphisms in $\Lambda C$.  If $F\colon C\to D$ is a functor, there is an induced functor $\Lambda F\colon \Lambda C\to \Lambda D$ defined by applying $F$. (Compare to \cref{lambda_of_a_composite}.)

  \begin{ex}
      For a group $G$, note that $BG$ is not monoidal, but we can extend this definition by dropping  the dualizability assumption.  In that case, the objects of $\Lambda BG$ are the elements of $G$.  The morphisms $g\to k$ are elements $h$ of $G$ so that $gh=hk$.
  \end{ex}
  In particular, the isomorphism classes of objects in $\Lambda BG$ correspond to conjugacy classes of elements of $G$.
  
  \begin{defn}[Compare to \cref{defn:bicat_trace}]\label{defn:bicat_trace_early}\cite[4.1]{Dold1978-ol}
If $f\colon X\to X$ is an endomorphism of a dualizable object in a symmetric monoidal category $C$, the {\bf trace}  of $f$, $\tr(f)$,  is the composite 
\[S\xrightarrow{\eta} X\wedge Y\xrightarrow{f\wedge \id}X\wedge Y\cong Y\wedge X\xrightarrow{\epsilon}S.\]  
\end{defn}
In $\Vect_k$ this is the usual definition in terms of sum of diagonal entries of a matrix representation. 

\begin{prop}[Compare to \cref{thm:character}]\label{thm:character_early} Let $C(S,S)$ be the endomorphisms of the unit object of $C$ regarded as a category with only identity morphisms.
If $C$ is symmetric monoidal, there is a functor 
\[\chi\colon \Lambda C\to C(S,S)\]
defined on objects by 
$\chi(A, f)=\tr(f).$
\end{prop}

The morphisms in \cref{Lambda_bicat_early} induce identity maps between the corresponding traces and so this definition on objects extends to morphisms.

\begin{defn}[Compare to \cref{lab:def:character}]\label{def:character_early}
    If $C$ is a symmetric monoidal category, the \textbf{character} of a representation $\rho\colon BG\to C$ where $\rho(\ast)$ is dualizable is the composite
    \[\Lambda (BG)\xto{\Lambda (\rho)}\Lambda (C)\xto{\chi} C(S,S).\]
\end{defn}
Mimicking classical notation, but not overloading $\chi$, we will write $\mathcal{X}_\rho$ for the character of a representation $\rho$.

\begin{ex}
    Let $\rho: BG \to \Vect_k$ be a linear representation where $\rho(\ast) = V$ is finite dimensional. Then  
    \[\Lambda \rho: \Lambda BG \to \Lambda \Vect_k\] 
    is given on objects by $(\Lambda\rho)(\ast, g) = (V,\rho(g))$.  A morphism $h$ from $g$ to $k$ is taken to $\rho(h)$ where 
    \begin{equation}\label{eq:1_cat_conj_inv}\xymatrix{V \ar[r]^{\rho(g)} \ar[d]_{\rho(h)} & V \ar[d]^{\rho(h)}\\
    V \ar[r]_{\rho(k)} & V}\end{equation}
    commutes. The character of $\rho$ takes $(V,g)$ to $\tr(\rho(g))$.  Since 
the diagram in \eqref{eq:1_cat_conj_inv} implies $\tr(\rho(g))=\tr(\rho(k))$, this defines a functor with domain $\Lambda BG$ and the character is conjugation invariant. 
\end{ex}

\begin{ex}[Compare to \cref{sec:restriction}]
\label{ex:class_restriction}
    If $\iota\colon H\to G$ is a group homomorphism, the {\bf restriction} of $\rho\colon BG\to C$, $R_{BG}^{BH}(\rho)$,  is the composite 
    \[BH\xto{B\iota}BG\xto{\rho}C.\]
    The character of the restriction of $\rho$ is 
    \[\Lambda BH\xto{\Lambda (\rho\circ B\iota)} \Lambda C\xto{\chi} C(S,S).\]
    Since $\Lambda(\rho\circ B\iota)=\Lambda(\rho)\circ \Lambda (B\iota)$ the character of the restriction is $\Lambda B\iota$ composed with the character of $\rho$. In particular, 
    \[ \mathcal{X}_{R_{BG}^{BH}(\rho)}=R_{\Lambda BG}^{\Lambda BH}(\mathcal{X}_{\rho}).\]
\end{ex}

    Induction is more elaborate to define.  For a group homomorphism $\iota\colon H\to G$ and an $H$-representation $\rho\colon BH\to C$, consider the under category 
    $\rho\downarrow (B\iota)^*$.  (Here we abuse notation and use $\rho$ to denote  the functor as in the following diagram.)
    % https://q.uiver.app/#q=WzAsMyxbMCwxLCJcXEZ1bihCRyxDKSJdLFsxLDEsIlxcRnVuKEJILEMpIl0sWzEsMCwiXFxhc3QiXSxbMCwxLCIoQlxcaW90YSleKiIsMl0sWzIsMSwiXFxyaG8iXV0=
\[\begin{tikzcd}
	& \ast \\
	{\Fun(BG,C)} & {\Fun(BH,C)}
	\arrow["\rho", from=1-2, to=2-2]
	\arrow["{(B\iota)^*}"', from=2-1, to=2-2]
\end{tikzcd}\]  The initial object of this category, denoted $I_{BH}^{BG}\rho$, is the usual {\bf induced representation} of $\rho$.  Note also that this definition applies to categories other than $BH$ and $BG$ and any functor between them -- so we can induce characters in addition to representations.

Focusing on the case of  $\rho\downarrow (B\iota)^*$, 
\begin{itemize}
\item the objects are pairs $(\phi, \theta)$ where $\phi$ is a functor $BG\to C$ and $\theta$ is a natural transformation from $\rho$ to 
\[BH\xrightarrow{B\iota}BG\xrightarrow{\phi}C.\]  
\item A morphism  $(\phi, \theta) \to  (\phi', \theta')$ is a natural transformation $\psi\colon \phi\to \phi'$ so that 
% https://q.uiver.app/#q=WzAsMyxbMSwxLCJDIl0sWzEsMCwiQkgiXSxbMCwxLCJCRyJdLFsxLDAsIlxccmhvIl0sWzIsMCwiXFxwaGkiLDEseyJjdXJ2ZSI6LTF9XSxbMSwyLCJCXFxpb3RhIiwyLHsiY3VydmUiOjJ9XSxbMiwwLCJcXHBoaSciLDIseyJjdXJ2ZSI6Mn1dLFsxLDIsIlxcdGhldGEiLDIseyJzaG9ydGVuIjp7InNvdXJjZSI6MjAsInRhcmdldCI6MjB9LCJsZXZlbCI6Mn1dLFs0LDYsIlxccHNpIiwwLHsic2hvcnRlbiI6eyJzb3VyY2UiOjIwLCJ0YXJnZXQiOjIwfX1dXQ==
\[\begin{tikzcd}
	& BH \\
	BG & C
	\arrow["{B\iota}"', curve={height=12pt}, from=1-2, to=2-1]
	\arrow["\theta"', shorten <=4pt, shorten >=4pt, Rightarrow, from=1-2, to=2-1]
	\arrow["\rho", from=1-2, to=2-2]
	\arrow[""{name=0, anchor=center, inner sep=0}, "\phi"{description}, curve={height=-6pt}, from=2-1, to=2-2]
	\arrow[""{name=1, anchor=center, inner sep=0}, "{\phi'}"', curve={height=12pt}, from=2-1, to=2-2]
	\arrow["\psi", shorten <=2pt, shorten >=2pt, Rightarrow, from=0, to=1]
\end{tikzcd}\]
is $\theta'$.  
\end{itemize}
Let $(\rho\downarrow (B\iota)^*)^d$ be the full subcategory of $\rho\downarrow (B\iota)^*$ whose objects are pairs $(\phi,\theta)$ where $\phi(\ast)$ is dualizable.

\begin{thm}[Compare to \cref{intro:induced}]\label{thm:intro:induced_classical}
    The character extends to a functor 
    \begin{equation}\label{eq:smc_character_functor}(\rho\downarrow (B\iota)^*)^d\to (\mathcal{X}_\rho
    \downarrow (\Lambda B\iota)^*).\end{equation}
If  $I_{BH}^{BG}(\rho)\in (\rho\downarrow (B\iota)^*)^d$, then 
    \[I_{\Lambda BH}^{\Lambda BG}(\mathcal{X}_\rho)=
    \mathcal{X}_{
    I_{BH}^{BG}(\rho)}.\]
\end{thm}

The dual statement for coinduction also holds. 

\begin{proof}

The objects of $(\mathcal{X}_\rho
    \downarrow (\Lambda B\iota)^*)$ are pairs $(\varphi, \vartheta)$ where $\varphi$ is a functor $\Lambda BG\to C(S,S)$ and $\vartheta$ is a natural transformation from $\mathcal{X}_\rho$ to 
\[\Lambda BH\xrightarrow{\Lambda B\iota}\Lambda BG\xrightarrow{\varphi}C(S,S).\]   Since $C(S,S)$ only has identity morphisms, $\vartheta$ reduces to equalities.  

The image of a pair $(\phi,\theta)$ under \eqref{eq:smc_character_functor} is $(\mathcal{X}_\phi, \chi\circ (\Lambda \theta) )$.  Note that the restriction on $\rho\downarrow (B\iota)^*$ implies  $\mathcal{X}_\phi$ is defined.  The image of a morphism is similar. 

The last statement follows from the observation that there is an induced map     \[I_{\Lambda BH}^{\Lambda BG}(\mathcal{X}_\rho)\to 
    \mathcal{X}_{
    I_{BH}^{BG}(\rho)}\]
    but all maps of this form must be identities since $C(S,S)$ only has identity maps.
\end{proof}
In particular, this recovers the classical statement about characters of induced representations for pairs of finite groups since the induced representation will be finite dimensional.

If we replace under categories by over categories we recover the corresponding classical statements about coinduction.

Stated very informally, the project of this paper is the following.
\begin{quote}
    \emph{There is a theory of characters, restriction, induction and conduction for 2-representations that very precisely parallels the story above.}
\end{quote}

\section{Shadows and traces in bicategories}\label{sec:shadows}

To extend the results of the previous section to  2-representations we recall the theory of  traces in bicategories with shadows \cite{p:thesis,Ponto_2012}.  This is a generalization of  Dold and Puppe's  trace \cite{Dold1978-ol} used in \cref{intro:classical}.

Let $\mathcal{B}$ be a bicategory with bicategorical composition denoted $\odot$ and unit one cells denoted $U_A$.  See \cite{leinster} for the definition of a bicategory and associated structure. 

\begin{defn}[{Compare to \cref{defn:dual_early}}] \cite[16.4.1]{may2004parametrized}\label{defn:dual}
    A 1-cell $M$ of $\mathcal{B}$ from $A$ to $B$ is \textbf{(right) dualizable} if there exists an object $M^\ast$, called its \textbf{dual}, and 2-cells \[U_A \overset{\eta}{\to} M \odot M^\ast \hspace{10mm} M^\ast\odot M \overset{\epsilon}{\to} U_B \] satisfying the triangle identities. The maps $\eta$ and $\epsilon$ are called \textbf{coevaluation} and \textbf{evaluation}, respectively.  
\end{defn}
If $\mathcal{B}$ is a bicategory with a single 0-cell this reduces to  \cref{defn:dual_early}.
 \begin{ex}
 In the bicategory $\Mod/\Ring$ 
 \begin{itemize}
     \item 0-cells are rings, 
     \item 1-cells are bimodules, and 
     \item 2-cells are bimodule homomorphisms.
 \end{itemize} The tensor product, $\otimes$, is the composition of 1-cells. 

 The dualizable 1-cells are $R$-$S$-bimodules $M$ so that $M$ is finitely generated and projective as a right $S$-module.
\end{ex}

\begin{ex}
    In the bicategory $2\Vect_k$
    \begin{itemize}
        \item 0-cells are symbols $[n]$, $n = 1, 2, 3, \dots$,
        \item a 1-cell from $[m]$ to $[n]$ is an $m\times n$ matrix of $k$-vector spaces,
        \item a 2-cells from $[A_{ij}]$ to $[B_{ij}]$ (both having the same size) is linear maps $\phi = \{\phi_{ij}: A_{ij} \to B_{ij}\}$.
    \end{itemize}
    Composition of 1-cells is given by \[(A\odot B)_{ij} = \underset{\ell}{\bigoplus} A_{i\ell}\otimes B_{\ell j}.\] 
    This is the special case of the $\Mat(\mathcal{V})$ construction from \cite{betti} where $\mathcal{V}$ is a monoidal category with small coproducts preserved by the tensor product.

    A 1-cell is dualizable if all the entries in the matrix are finite dimensional vectors spaces. The dual of a 1-cell $[A_{ij}]$ is $[A^*_{ji}]$, the $ij$ entry of the coevaluation and evaluation is 0 if $i\neq j$, the coevaluation in the $ii$ entry 
     \[k\to \bigoplus_{m} A_{im}\otimes A^*_{im}\]
    is the sum of the coevaluations 
    \[k\to A_{im}\otimes A^*_{im}.\]
    The evaluation is similar. 
\end{ex}

Other useful sources of dualizable 1-cells are the bicategory of profunctors \cite[Chapter 7]{Borceux} and the homotopy bicategory of parameterized spectra \cite{may2004parametrized}.

To define the trace we need to be able to compose the coevaluation and evaluation, but in a bicategory these are not composable.  We resolve this with the addition of a shadow.
\begin{defn}\cite[Definition 4.1]{Ponto_2012}\label{defn:shadow}
    Let $\B$ be a bicategory. 
    A \textbf{shadow functor} for $\B$ consists of functors \[\sh{\text{-}}: \B(R, R) \to \mathbf{T}\] for each object $R$ of $\B$ and some fixed category $\mathbf{T}$, equipped with a natural isomorphism \[\theta: \sh{M \odot N} \overset{\cong}{\to} \sh{N \odot M} \] for $M: R \to S$ and $N: S \to R$ such that the following diagrams commute whenever they make sense:

% https://q.uiver.app/#q=WzAsMyxbMCwxLCJcXHNoe01cXG9kb3QgTlxcb2RvdCBQfSJdLFsyLDEsIlxcc2h7UFxcb2RvdCBNXFxvZG90IE59Il0sWzEsMCwiXFxzaHtOXFxvZG90IFBcXG9kb3QgTX0iXSxbMiwxLCJcXHRoZXRhIl0sWzAsMSwiXFx0aGV0YSJdLFswLDIsIlxcdGhldGEiXV0=&macro_url=https%3A%2F%2Fwww.overleaf.com%2Fproject%2F64c9c0fcea71c453e68a663b
\[\begin{tikzcd}[column sep = .1in]
	& {\sh{N\odot P\odot M}} \\
	{\sh{M\odot N\odot P}} && {\sh{P\odot M\odot N}}
	\arrow["\theta", from=1-2, to=2-3]
	\arrow["\theta", from=2-1, to=1-2]
	\arrow["\theta", from=2-1, to=2-3]
\end{tikzcd}
%
% https://q.uiver.app/#q=WzAsNCxbMCwwLCJcXHNoe01cXG9kb3QgVV9TfSJdLFsxLDAsIlxcc2h7VV9SXFxvZG90IE19Il0sWzIsMCwiXFxzaHtNXFxvZG90IFVfU30iXSxbMSwxLCJcXHNoe019Il0sWzAsMywiXFxzaHtyfSIsMl0sWzEsMywiXFxzaHtcXGVsbH0iXSxbMiwzLCJcXHNoe3J9Il0sWzAsMSwiXFx0aGV0YSJdLFsxLDIsIlxcdGhldGEiXV0=
\begin{tikzcd}
	{\sh{M\odot U_S}} & {\sh{U_R\odot M}} & {\sh{M\odot U_S}} \\
	& {\sh{M}}
	\arrow["{\sh{r}}"', from=1-1, to=2-2]
	\arrow["{\sh{\ell}}", from=1-2, to=2-2]
	\arrow["{\sh{r}}", from=1-3, to=2-2]
	\arrow["\theta", from=1-1, to=1-2]
	\arrow["\theta", from=1-2, to=1-3]
\end{tikzcd}\]
\end{defn}
\begin{rmk*}  We use the coherence results from \cite{bicat_coherence,MP:coherence} to omit structure maps, such as the associativity maps above, in bicategories (with and without shadows). 
\end{rmk*}

If $C$ is a symmetric monoidal category regarded as a bicategory with a single 0-cell the identity is a shadow.  For the bicategory $\Mod/\Ring$, the 0-th Hochschild homology is a shadow. 

\begin{ex}\label{shadow2Vect}
    Given a 2-matrix $A = [A_{ij}]$ in  the bicategory $2\Vect_k$, the shadow of $A$ is the trace of $A$. That is, \[\sh{A} = \bigoplus_i A_{ii}.\] In general,  $AB \neq BA$ for matrices $A$ and $B$,  however, the main diagonal remains unchanged. 
\end{ex}

\begin{defn}[Compare to \cref{defn:bicat_trace_early}]\cite[Definition 5.1]{Ponto_2012}\label{defn:bicat_trace}
    Let $\B$ be a bicategory with a shadow functor and $M$ a dualizable 1-cell of $\B$. The \textbf{bicategorical trace} of a 2-cell $f: Q\odot M \Rightarrow M \odot P$ is the composite: \[\xymatrix{\sh{Q} \ar[r]^-{\sh{\id\odot \eta}} & \sh{Q \odot M \odot M^\ast} \ar[r]^-{\sh{f\odot \id}} & \sh{M \odot P \odot M^\ast} \ar[r]^\theta & \sh{M^\ast \odot M\odot P} \ar[r]^-{\sh{\epsilon\odot \id}} & \sh{P}}.\]
\end{defn}

If $C$ is a symmetric monoidal category regarded as a bicategory with a single 0-cell and the identity as the shadow this reduces to Dold and Puppe's definition in \cite{Dold1978-ol}.  

\begin{ex}In the bicategory $\Mod/\Ring$ with the 0-th Hochschild homology as the shadow, this trace is a generalization of the Hattori-Stallings trace \cite{stallings}. 
\end{ex}

\begin{ex}
\label{ex:2_vect_trace}
A 2-cell $f\colon [A_{ij}]\odot [B_{ij}]\to [B_{ij}]\odot [D_{ij}]$ in  $2\Vect_k$ 
is  transformations
\[f_{ij}\colon \bigoplus_k A_{ik}\otimes B_{kj}
\to \bigoplus_{k}B_{ik}\otimes D_{kj}.\]
Consider the composite 
\begin{equation}\label{eq:2_vect_mor}A_{ii}\otimes B_{ij}\xto{\mathrm{inc}_i}
\bigoplus_k A_{ik}\otimes B_{kj}
\xto{f_{ij}} \bigoplus_{k}B_{ik}\otimes D_{kj}
\xto{\mathrm{proj}_j} 
B_{ij}\otimes D_{jj}\end{equation}
where the first map is the inclusion and the last is the projection.

Using the shadow from \cref{shadow2Vect}, the trace $f$ is a linear transformation 
\[\bigoplus_iA_{ii}\xto{\tr(f)} \bigoplus_{j}D_{jj}.\]
It is determined by the composites with each inclusion $A_{ii}\to \bigoplus_iA_{ii}$ and projection $\bigoplus_{j}D_{jj}\to D_{jj}$.

To simplify the identification of the trace  we first show that the inclusion and projection are traces. Let  $R_i(A)$ be a row vector where the only nonzero entry is $A$ in the $i^{th}$ entry and $C_j(D)$ is similar for columns. 
Consider the 2-cell
\[
R_i(A_{ii})\to 
\left[\begin{array}{ccccccc} A_{i1}&\cdots &A_{i(i-1)}&A_{ii} &A_{i(i+1)}&\cdots &A_{in}
\end{array}\right].\]
that includes the $i^{th}$ entry in a row vector of vector spaces.  This is a 2-cell
\begin{equation}\label{eq:row:inclusion} A_{ii}\odot R_i(k)
\to 
R_i(k)\odot 
[A_{ij}]\end{equation}
The vector 
$R_i(k)$
is dualizable and the trace of \eqref{eq:row:inclusion} is the inclusion $A_{ii}\to \bigoplus_{i}A_{ii}$.

In the same way  the 2-cell
\[\left[\begin{array}{c} D_{1j}\\\vdots \\D_{(j-1)j}\\D_{jj} \\D_{(j+1)j}\\\vdots\\D_{nj}
\end{array}\right]\to 
C_j(D_{jj})
\]
that projects to  the $j^{th}$ entry in a column vector of vector spaces is a map 
\begin{equation}\label{eq:column:projection} 
[D_{ij}]\odot C_j(k)
\to 
C_j(k)\odot D_{jj}
\end{equation}
The column vector $C_j(k)$
is dualizable and the trace of \eqref{eq:column:projection} is the projection $\bigoplus_{j}D_{jj}\to D_{jj}$.

Since the first and last maps the composite 
\begin{equation}\label{eq:2vect:include:project}
    A_{ii}\to \bigoplus_iA_{ii}\xto{\tr(f)} \bigoplus_{j}D_{jj}\to D_{jj}
\end{equation} are traces, \cite[7.5]{Ponto_2012} implies 
the composite in \eqref{eq:2vect:include:project} is the trace of the left composite in \cref{fig:trace:2:vect}.   The right composite in \cref{fig:trace:2:vect} is \eqref{eq:2_vect_mor}
and the composite \eqref{eq:2vect:include:project} is 
the trace of \eqref{eq:2_vect_mor}.

\begin{figure}
% https://q.uiver.app/#q=WzAsOCxbMCwwLCJBX3tpaX1cXG9kb3QgUl9pKGspXFxvZG90IFtCX3tpan1dXFxvZG90IENfaihrKSJdLFswLDEsIlJfaShrKVxcb2RvdCBbQV97aWp9XVxcb2RvdCBbQl97aWp9XVxcb2RvdCBDX2ooaykiXSxbMCwyLCJSX2koaylcXG9kb3QgW0Jfe2lqfV1cXG9kb3QgW0Rfe2lqfV1cXG9kb3QgQ19qKGspIl0sWzAsMywiUl9pKGspXFxvZG90IFtCX3tpan1dXFxvZG90IENfaihrKVxcb2RvdCBEX3tqan0iXSxbMSwwLCJBX3tpaX1cXG90aW1lcyBCX3tpan0iXSxbMSwxLCJcXGJpZ29wbHVzX2sgKEFfe2lrfVxcb3RpbWVzIEJfe2tqfSkiXSxbMSwyLCJcXGJpZ29wbHVzX2soQl97aWt9XFxvdGltZXMgRF97a2p9KSJdLFsxLDMsIkJfe2lqfVxcb3RpbWVzIERfe2pqfSJdLFs2LDcsIlxcbWF0aHJte3Byb2p9X2oiXSxbMiw2LCJcXHNpbSJdLFszLDcsIlxcc2ltIl0sWzIsMywiMVxcb2RvdCAxXFxvZG90IFxcZXFyZWZ7ZXE6Y29sdW1uOnByb2plY3Rpb259Il0sWzAsNCwiXFxzaW0iXSxbMSw1LCJcXHNpbSJdLFswLDEsIlxcZXFyZWZ7ZXE6cm93OmluY2x1c2lvbn1cXG9kb3QgMVxcb2RvdDEiXSxbNCw1LCJcXG1hdGhybXtpbmN9X2kiXSxbMSwyLCIxXFxvZG90IGZcXG9kb3QgMSJdLFs1LDYsImZfe2lqfSJdXQ==&macro_url=https%3A%2F%2Fwww.overleaf.com%2Fproject%2F64c9c0fcea71c453e68a663b
\[\begin{tikzcd}
	{A_{ii}\odot R_i(k)\odot [B_{ij}]\odot C_j(k)} & {A_{ii}\otimes B_{ij}} \\
	{R_i(k)\odot [A_{ij}]\odot [B_{ij}]\odot C_j(k)} & {\bigoplus_k (A_{ik}\otimes B_{kj})} \\
	{R_i(k)\odot [B_{ij}]\odot [D_{ij}]\odot C_j(k)} & {\bigoplus_k(B_{ik}\otimes D_{kj})} \\
	{R_i(k)\odot [B_{ij}]\odot C_j(k)\odot D_{jj}} & {B_{ij}\otimes D_{jj}}
	\arrow["\sim", from=1-1, to=1-2]
	\arrow["{\eqref{eq:row:inclusion}\odot 1\odot1}", from=1-1, to=2-1]
	\arrow["{\mathrm{inc}_i}", from=1-2, to=2-2]
	\arrow["\sim", from=2-1, to=2-2]
	\arrow["{1\odot f\odot 1}", from=2-1, to=3-1]
	\arrow["{f_{ij}}", from=2-2, to=3-2]
	\arrow["\sim", from=3-1, to=3-2]
	\arrow["{1\odot 1\odot \eqref{eq:column:projection}}", from=3-1, to=4-1]
	\arrow["{\mathrm{proj}_j}", from=3-2, to=4-2]
	\arrow["\sim", from=4-1, to=4-2]
\end{tikzcd}\]
    \caption{Trace in \cref{ex:2_vect_trace}}\label{fig:trace:2:vect}
\end{figure}
\end{ex}

\section{Statements of the main results}\label{sec:main}

A significant amount of the effort in this paper is in careful checking of diagrams.  In this section we state the main results in a level of detail we hope will allow the reader to appreciate the forest without feeling they are missing out too much on the trees.  (The trees can be found in \cref{sec:remnant,sec:remnant_functors,sec:traces}.)  This section is parallel to \cref{intro:classical}.

We start with the generalization of the endomorphism category.

\begin{restatable}[Compare to \cref{Lambda_bicat_early}]{thm}{restateLambdabicat}
\label{Lambda_bicat}
    Given a bicategory $\mathcal{B}$, there is a bicategory $\Lambda \mathcal{B}$ where 
    \begin{itemize}
        \item the 0-cells are the endomorphism 1-cells of $\mathcal{B}$,
        \item the 1-cells $P\to Q$ are pairs $(M,\psi)$ where $M$ is a dualizable 1-cell in $\mathcal{B}$ and 
        $\psi$ is a 2-cell 
        \[P\odot M\Rightarrow M\odot Q,\]
        and 
        \item the 2-cells $(M,\psi)\Rightarrow (M',\psi')$ are isomorphism 2-cells $\beta\colon M\Rightarrow M'$ of $\mathcal{B}$ so that the diagram 
          \[\xymatrix{P\odot M \ar@{=>}[d]_{\id\odot\beta} \ar@{=>}[r]^\psi & M\odot Q \ar@{=>}[d]^{\beta\odot\id}\\
    P\odot M' \ar@{=>}[r]_{\psi'} & M'\odot Q}\]
    commutes.
    \end{itemize}
\end{restatable}
We call this the {\bf \remnant{}  
 bicategory} of $\mathcal{B}$. 
See \cref{sec:remnant} for the proof of \cref{Lambda_bicat}.

 The notation of $\Lambda$ is inspired by the following example. 

 \begin{ex}\label{ex:inertia}  Let $G$ be a group and $BG$ be the bicategory with a single 0-cell, $G$ as 1-cells, and only identity 2-cells.  
    The objects of $\Lambda BG$ are the elements of $G$.  A 1-cell in $\Lambda BG$ from $g$ to $h$ is an element $k$ of $G$ so that 
    \[gk=kh.\]  (Note that all 1-cells are dualizable in $BG$.)  The only 2-cells are identities.  This is the \textbf{inertia groupoid} of $G$.
    
    In contrast to \cref{intro:classical}, here we can apply the definition of $\Lambda$ directly to $BG$ since it is a bicategory.
\end{ex}

One feature of this bicategory that is very different from its classical inspiration is the ``twisted'' endomorphisms rather than the more familiar true endomorphisms.  \cref{ex:inertia} is a good source of inspiration for this twisting where it reflects the conjugation invariance of the character.

\begin{restatable}{thm}{lmabdaofacomposite}
 \label{lambda_of_a_composite}
    A strong 2-functor $F\colon \B\to \mathcal{C}$ induces a strong 2-functor 
    \[\Lambda F\colon \Lambda \B\to \Lambda \mathcal{C} \] 
where \begin{itemize}
\item on 0-cells $\Lambda F(P)=F(P)$,
\item on 1-cells $\Lambda F(M,\psi)=(F(M), F(P)\odot F(M)\xto{\phi} F(P\odot M)\xto{F(\psi)} F(M\odot Q)\xto{\phi^{-1}} F(M)\odot F(Q))$, and 
\item on 2-cells $\Lambda F(\beta)=F(\beta)$.
\end{itemize}

    If $F\colon \B\to \mathcal{C}$ and $G\colon \mathcal{C}\to \mathcal{D}$ are strong 2-functors, 
\[\Lambda(F\circ G)=(\Lambda F)\circ (\Lambda G).\]
\end{restatable}
See  \cref{sec:remnant_functors} for the proof. 

The primary example of interest is {\bf 2-representations}, that is  strong 2-functors 
\[\rho\colon BG\to \mathcal{C}.\]
A 2-representation consists of  
\begin{enumerate}
        \item an object $V$ of $\cc$,
        \item for each element $g\in G$, a dualizable 1-cell $V\xto{\rho(g)}V$,
        \item for any pair of elements $g, h$ of $G$, a natural 2-isomorphism

        \begin{equation}\label{phi_2_rep_pair}\phi_{g,h}: \rho(g) \odot \rho(h) \overset{\cong}{\implies} \rho(gh),\end{equation} and
        \item\label{unit_2_rep}  if $e$ is the identity element of $G$, a natural 2-isomorphism
        \begin{equation}\label{phi_2_rep_unit}\phi_e: \rho(e) \overset{\cong}{\implies} U_V,\end{equation}
    \end{enumerate}
        satisfying unit and associativity conditions.  On 0- and 1-cells, the functor  $\Lambda\rho$ is given by 
        \begin{enumerate}
            \item a 1-cell $V\xto{\rho(g)}V$ in $\mathcal{C}$ for each $g\in G$
            \item for each equality $gh=hk$ in $G$, a 2-cell 
            \begin{equation}\label{eq:LambdaGC:2-cell}\rho(g)\odot \rho(h)\Rightarrow \rho(h)\odot \rho(k)\end{equation}
            in $\mathcal{C}$.
        \end{enumerate}
    The remaining structure follows from that in $\mathcal{C}$ and $\rho$.
    
    \begin{ex}One of the most immediate sources of 2-representations are classical linear representations. Suppose $\rho$ is a linear representation of a group $G$  and, for all $g\in G$, all entries are 0 or 1 in  a matrix representation for $\rho(g)$.  Then  $\rho$ defines a 2-$G$-representation in 2-Vect where the entries with 0 are replaced by the zero dimensional vectors space and those with entry 1 are replaced by a 1-dimensional vector space. 
\end{ex}

Base change objects provide a very different source of 2-representations.
\begin{ex}\label{ring2rep}
Let $\psi\colon G\to \Aut(R)$ be a group homomorphism from a group $G$ to the automorphisms of a ring $R$.  
Let  $R_{\psi(g)}$ be the set $R$ with the usual left multiplication of $R$ on itself and right action twisted by $\psi(g)$.  This is a \emph{base change} 1-cell in   
$\Mod/\Ring$  and satisfies \eqref{phi_2_rep_pair} and \eqref{phi_2_rep_unit}.  It is dualizable by \cite[16.4.2]{may2004parametrized} and so the assignment of $R_{\psi(g)}$ to $g$ defines a $G$-2-representation in $\Mod/\Ring$.
\end{ex}

The following example has the same foundation but it technically more demanding. 

\begin{ex}\label{top2rep}
Let $\psi\colon G\to \Aut(X)$ be a group homomorphism from a group $G$ to the automorphisms of a topological space $X$.  The base change 1-cells 
\[X_{\psi(g)}\coloneqq (\id\times \psi(g))^*U_X\]
 \cite[17.2.1]{may2004parametrized} define 1-cells in the parametrized stable homotopy category and satisfy \eqref{phi_2_rep_pair} by \cite[17.2.6]{may2004parametrized} and \eqref{phi_2_rep_unit} by definition.  They are dualizable by \cite[17.3.1]{may2004parametrized} and so this assignment defines a $G$-2-representation in the parameterized stable homotopy category.
\end{ex}

The last step before defining characters for these 2-representations is to extend trace to a functor.
    
\begin{restatable}[Compare to \cref{thm:character_early}] {thm}{thmcharacter}
\label{thm:character}
Let $\B$ be a bicategory with shadow functors 
\[\sh{\text{-}}: \B(R, R) \to \mathbf{T}.\]
Let $\T$ be $\mathbf{T}$ regarded as a 2-category with only identity 2-cells. 
There is a strict 2-functor 
$\rchi: \Lambda \B \to \T$ 
defined 
by \begin{align*} 
\rchi(P) &= \sh{P},
\\ 
\rchi[(N, \gamma)] &= \tr(\gamma),\ \text{and} 
\\ 
\rchi(\ltwo{\psi}) &= \id_{\tr(\gamma)}. \end{align*} 
\end{restatable}

See \cref{sec:traces} for the proof.

\begin{restatable}[Compare to \cref{def:character_early}]{defn}{defcharacter}
\label{lab:def:character}
    Let $\B$ be a bicategory with shadow $\sh{\,}$ valued in $\mathbf{T}$ and 
    \[\rho\colon \mathcal{C}\to \B\] be a strong 2-functor. 
    The \textbf{character} of $\rho$ is the composite
    \[\Lambda \mathcal{C}\xto{\Lambda \rho} \Lambda \B\xto{\chi}\T,\] where $\Lambda\rho$ is the induced strong functor from \cref{lambda_of_a_composite}
and $\rchi$ is the strict 2-functor from \cref{thm:character}.
\end{restatable}

Following \cref{def:character_early}, we denote this composite $\mathcal{X}_\rho$.

Explicitly, for a strong 2-functor $\rho\colon BG\to \mathcal{B}$, the character is the functor 
\[\Lambda BG\to \mathcal{T}\] given by 
\begin{enumerate}
    \item for each $g\in G$, the object $\chi\circ (\Lambda \rho)(g)=\sh{\rho(g)}$ in ${T}$, and 
    \item each equality $gh=hk$ in $G$, the morphism  
    \begin{equation} \sh{\rho(g)}\xto{\tr\eqref{eq:LambdaGC:2-cell}}\sh{\rho(k)}
    \label{eq:character_on_morphism}
    \end{equation}
            in $\mathcal{T}$.
\end{enumerate}

\begin{ex} If we use the 0th Hochschild homology as the shadow in \cref{ring2rep}, then the image of an equality $gh=hk$ is the map 
\[HH_0(R, R_{\psi(g)})\to HH_0(R, R_{\psi(k)})\]
that sends a class $[r]$ to $[hr]$.
\end{ex}

\begin{ex}
Returning to \cref{top2rep}, $\sh{X_{\psi(g)}}$ is equivalent to $\Lambda ^{\psi(g)}X$, the $\psi(g)$ twisted loops in $X$.    The image of an equality of group elements $gh=hk$ is the map 
\[\Lambda ^{\psi(g)}X\to \Lambda ^{\psi(k)}X\]
that takes a twisted loop $\gamma $ in $\Lambda ^{\psi(g)}X$ to $\psi(h)(\gamma)$ in $\Lambda ^{\psi(k)}X$.

Considering the special case of the dihedral group of order 6 acting on the circle, the relation $rs=sr^{-1}$ gives a morphism 
\[\Lambda^rS^1\to \Lambda ^{r^{-1}}S^1\]
by reflection.
\end{ex}

\subsection{Restriction of representations}\label{sec:restriction}\label{charrestrict}
    For a functor $\iota\colon \mathcal{C}\to \mathcal{D}$  and representation $\rho\colon \mathcal{D}\to \B$, the \textbf{restriction of $\rho$ along $\iota$}, $R_\mathcal{D}^\mathcal{C}(\rho)$,  is the composite 
    \[\mathcal{C}\xto{\iota}\mathcal{D} \xto{\rho} \B.\]
The character of the restriction is 
    \[\Lambda \mathcal{C} \xto{\Lambda (\rho\circ \iota)} \Lambda \B\xto{\rchi} \T.\]
By \cref{lambda_of_a_composite}, this is 
        \[\Lambda \mathcal{C} \xto{\Lambda (\iota)} \Lambda \mathcal{D}\xto{\Lambda \rho}\Lambda \B\xto{\rchi} \T\]
        showing 
        $R_{\Lambda \mathcal{D}}^{\Lambda \mathcal{C}}(\mathcal{X}_\rho)=\mathcal{X}_{R_\mathcal{D}^\mathcal{C}(\rho)}$,
    generalizing the familiar result from classical representation theory.

\subsection{Induced and coinduced representations}
Unfortunately induced and coinduced representations are more complicated to define than restriction.  We follow \cref{intro:classical} and  use comma 2-categories to give a fairly straightforward, if notationally elaborate, definition.  

Let $\mathcal{C}, \mathcal{D}, \mathcal{B}$ be bicategories and  $\iota\colon \mathcal{C}\to \mathcal{D}$  and $\rhov\colon \mathcal{C}\to \mathcal{B}$ be strong 2-functors.  Let $\Fun(\mathcal{D},\mathcal{B})$ denote the bicategory of strong 2-functors, strong transformations, and modifications from $\mathcal{D}$ to $\mathcal{B}$.

   The functor $\iota$ induces a strong 2-functor
\[\iota^*\colon 
\Fun(\mathcal{D}, \mathcal{B})\to \Fun(\mathcal{C}, \mathcal{B}).\]
We abuse notation and let $\rhov$ denote the functor 
\[\ast\to \Fun(\mathcal{C}, \mathcal{B})\] that picks out the strong functor $\rhov$ for the single 0-cell and identity 1- and 2-cells. Let $\rhov\downarrow \iota^*$ be the 2-comma category \cite[I.2.5]{gray} of $\iota^*$ and $\rhov$.  

\begin{defn}\label{def:induced_rep}The {\bf induced representation} of $\rhov$ to $\mathcal{D}$, denoted $I_\mathcal{C}^\mathcal{D}\rhov$, is the initial 0-cell of $\rhov\downarrow \iota^*$. 
\end{defn}

The 0-cells of $\rhov\downarrow \iota^*$ are pairs $(\rhow, \epsilon)$ consisting of a  functor $\rhow\colon \mathcal{D}\to \mathcal{B}$ and a transformation of strong 2-functors 
\[\epsilon\colon \rhov\to \iota^*\rhow.\]
As in the classical case (\cref{thm:intro:induced_classical}), we  restrict attention to a subcategory of $\rhov\downarrow \iota^*$ to have compatibility between induction and characters.  We denote this subbicategory by $(\rhov\downarrow \iota^*)^d$and the definition can be found in  \cref{def:conditions_for_induction}. 
We 
will   abuse notation and also  use $I_{\mathcal{C}}^\mathcal{D}\rhov$ to indicate the functor as well as the 0-cell in $\rhov\downarrow \iota^*$.

\begin{restatable}{thm}{thm:intro:induced}
\label{intro:induced}  If $I_\mathcal{C}^\mathcal{D}(\rhov)$  is a 0-cell in $(\rhov\downarrow \iota^*)^d$,
    there is a (unique) 1-cell 
        \[I_{\Lambda \mathcal{C}}^{\Lambda \mathcal{D}} (\mathcal{X}_{\rhov})\to \mathcal{X}_{I_\mathcal{C}^\mathcal{D}(\rhov)}.\]
\end{restatable}

This result is an immediate consequence of \cref{thm:trace_induced} that shows the character induces a functor on the relevant comma categories.

There are three other combinations analogous to \cref{intro:induced}.  The first variation to consider is the coinduced representation. 

\begin{defn}
The {\bf coinduced representation} of $\rhov$ to $\mathcal{D}$, denoted $C_\mathcal{C}^\mathcal{D}\rhov$, is the terminal 0-cell of $\iota^*\downarrow \rhov$. 
\end{defn}

The proof of \cref{intro:induced} can be minimally modified to show there is a (unique) 1-cell 
    \[\mathcal{X}_{C_\mathcal{C}^\mathcal{D}(\rhov)}\to 
     C_{\Lambda \mathcal{C}}^{\Lambda \mathcal{D}} (\mathcal{X}_{\rhov})\]
     if $C_\mathcal{C}^\mathcal{D}(\rhov)$ is a 0-cell in the analog of $(\rhov\downarrow \iota^*)^d$.
     
In the two remaining variations of  \cref{intro:induced} we replace the shadows and traces of \cite{p:thesis} with the coshadows and cotraces of \cite{barhite2023bicategorical}.  Not only is this change a straightforward formal modification, it reflects the motivating example in \cite{GK}.  See \cref{sec:coshadows} for a discussion of this structure.

\subsection{Joint character}
In this subsection we require that the shadow takes values in a symmetric monoidal category $(\mathbf{T}, \otimes, S)$. 

\begin{defn}\label{def:repdual}
For a $G$-2-representation $\rho$, we say $g\in G$ is {\bf \repdual} relative to $\rho$ if $\sh{\rho(g)}$ is dualizable in $\mathbf{T}$.
\end{defn}

The bicategory of $2\Vect_k$  provides useful examples here.  For any $G$-2-representation in $2\Vect_k$, a group element $g\in G$ is {\repdual} if the entries of the  matrix $\rho(g)$ are finite dimensional vector spaces since $\sh{\rho(g)}$ is then a finite sum of finite dimensional vector spaces.   On the other hand, the bicategory of parameterized spectra is not similarly generous with illuminating examples since the shadow produces loop spaces.

\begin{defn}
If $gh=hg$, and $g$ is \repdual, the {\bf joint trace} of $(g,h)$ relative to a representation $\rho$ is the symmetric monoidal trace of 
\eqref{eq:character_on_morphism}.

\end{defn}

Then the following definition is motivated by \cref{lab:def:character}.

\begin{defn}\label{def:joint_character}
If all elements of $G$ are {\repdual} with respect to a 2-representation $\rho$, the {\bf joint character}
 is the functor
\[\Lambda(\Lambda BG)\to \mathbf{T}(S,S)\]
where the image of $(g,h)$ is the joint trace of $(g,h)$ relative to $\rho$.
\end{defn}

The 0-cells of  $\Lambda(\Lambda BG)$ are pairs $(g,h)$ so that $gh=hg$.  (More carefully, the $g$ is implicit but since $h$ does not carry the data of its source and target, it is useful to keep the $g$ as part of the notation.)  The 1-cells from $(g,h)$ to $(g',h')$ are group elements $k$ so that,
\begin{itemize}
\item $gk=kg'$, so that $k$ is a 1-cell in $\Lambda BG$ from $g$ to $g'$, and, 
\item   $hk=kh'$ so that $k$ satisfies the equality in the definition of the 1-cell in \cref{Lambda_bicat}.  
\end{itemize} (In that statement it is a 2-cell, but there are only identity 2-cells in this case.)  Since there are only identity 2-cells in $\Lambda(\Lambda BG)$, the joint character is a \emph{function} from pairs of commuting elements up simultaneous conjugation to the set $\mathbf{T}(S,S)$.

At this point  these ideas connect to those in \cite{BZN,cp:iterated}.  Suppose $\mathcal{B}$ is a monoidal bicategory with monoidal product $\otimes$ and monoidal unit $I$.  
\begin{defn}\cite[6.4, 6.16]{cp:iterated} \label{defn:2_dual} A 0-cell $A$ is {\bf 2-dualizable} if there is a 0-cell $A^\vee$, left dualizable 1-cells $C\in \mathcal{B}(I, A\otimes A^\vee)$ and $E\in \mathcal{B}(A^\vee\otimes A,I)$, and invertible 2-cells for the triangle identities. 
\end{defn}
By \cite[6.18]{cp:iterated} a monoidal bicategory where all 0-cells are 2-dualizable has a canonical shadow valued in a symmetric monoidal category.

\begin{ex}
Recall that $\mathcal{T}$ is a category $\mathbf{T}$ regarded as a bicategory with only identity 2-cells. If $\mathbf{T}$ is symmetric monoidal and all objects of $\mathbf{T}$ are dualizable then all 0-cells of $\mathcal{T}$ are 2-dualizable and the symmetric monoidal trace is a shadow on $\mathcal{T}$.  Then the joint character in \cref{def:joint_character} is applying the character construction of \cref{lab:def:character} twice. 
\end{ex}

With this example we now return to the compatibility with restriction, induction, and coinduction.

\begin{thm}
If all elements of $G$ are {\repdual} with respect to a 2-representation $\rho$, then 
\[R_{\Lambda(\Lambda BG)}^{\Lambda(\Lambda BH)} \mathcal{X}_{\mathcal{X}_\rho}=\mathcal{X}_{\mathcal{X}_{R_{BG}^{BH}} \rho}
\quad \quad I_{\Lambda(\Lambda BH)}^{\Lambda(\Lambda BG)} \mathcal{X}_{\mathcal{X}_\rho}=\mathcal{X}_{\mathcal{X}_{I_{BH}^{BG}} \rho}
\quad \quad  C_{\Lambda(\Lambda BH)}^{\Lambda(\Lambda BG)} \mathcal{X}_{\mathcal{X}_\rho}=\mathcal{X}_{\mathcal{X}_{C_{BH}^{BG}} \rho}\]
\end{thm}

\begin{proof}
The case of restriction follows from \cref{ex:class_restriction,sec:restriction}.  

By \cref{intro:induced}, we have 
$ I_{\Lambda BH}^{\Lambda BG} \mathcal{X}_\rho\xrightarrow{\exists!}\mathcal{X}_{I_{BH}^{BG}} \rho$.  Taking the character  of both sides and noting that, as in the proof of \cref{thm:intro:induced_classical}, morphisms become equalities we have 
\[ \mathcal{X}_{I_{\Lambda BH}^{\Lambda BG} \mathcal{X}_\rho}=\mathcal{X}_{\mathcal{X}_{I_{BH}^{BG}} \rho}.\]
By \cref{thm:intro:induced_classical}, 
\[ I_{\Lambda (\Lambda BH)}^{\Lambda(\Lambda BG)} \mathcal{X}_{\mathcal{X}_\rho}= \mathcal{X}_{I_{\Lambda BH}^{\Lambda BG} \mathcal{X}_\rho}\]
completing the proof for induction.  The case for coinduction is similar. 
\end{proof}

We also note the  following result which is an immediate consequence of the main results of \cite{BZN,cp:iterated}. 
\begin{thm}\label{thm:ind_of_order}
For a 2-representation in a monoidal bicategory where all 0-cells are dualizable, all group elements are {\repdual} and the joint character is independent of order of the commuting elements. 
\end{thm}

\subsection{Conjectural generalizations}
We expect the bicategory $\Lambda \mathcal{B}$, functor $\Lambda$ and the trace functor are fragments of a larger structure, but we regard this largely as motivation since a tricategory is a fairly unwieldy structure.  Our conjectural statements are the following. 
\begin{conj}[Compare to \cref{Lambda_bicat}]\label{conjecture_tricat}
    There is a tricategory $\Lambda_3\B$ where the 
    \begin{itemize}
        \item 0-cells: the 0-cells of $\B$.
        \item 1-cells $R$ to $S$: the 1-cells of $\B$ from $R$ to $S$.
        \item 2-cells $P$ to $Q$: pairs $(M,\beta)$ where 
        \[\beta\colon P\odot M\Rightarrow M\odot Q\]
        is a 2-cell in $\B$.
        \item 3-cells $(M,\beta)$ to $(N,\gamma)$ are isomorphism 2-cells $\psi\colon M\to N$ so that 
            \[\xymatrix{P\odot M \ar@{=>}[d]_{\id\odot\psi} \ar@{=>}[r]^\beta & M\odot Q \ar@{=>}[d]^{\psi\odot\id}\\
    P\odot N \ar@{=>}[r]_\gamma & N\odot Q}\]
    commutes.
    \end{itemize}
\end{conj}

\begin{conj}[Compare to \cref{lambda_of_a_composite}]\label{conjecture_tricat_lambda}
    The construction of $\Lambda$ on bicategories and strong functors of bicategories extends to a (appropriately strong) endofunctor on the tricategory of bicategories. 
\end{conj}

\begin{conj}[Compare to \cref{thm:character}]
    The functor in \cref{thm:character} is an example of a shadow on a tricategory (when that is defined!).
\end{conj}

\section{Twisted endomorphism bicategories}\label{sec:remnant}
In this section we verify the existence of the endomorphism bicategory $\Lambda \mathcal{B}$ as in the following statement.

\restateLambdabicat*

We note here that the assumption that some of the 1-cells are dualizable and some 2-cells are isomorphisms are forward looking assumptions for \cref{sec:traces}.  They are not required to define a bicategory and lead to more work in this section.

We start by verifying that $\Lambda \mathcal{B}(P,Q)$ is a category for endomorphism 1-cells $P$ and $Q$ of $\mathcal{B}$.

\begin{lem}\label{LambdaBPQ_Cat}
    For endomorphism 1-cells $P$ and $Q$  in $\B$, $\Lambda \B(P, Q)$ is a category.
\end{lem}

\begin{proof}
The composition of morphisms in $\Lambda \B(P,Q)$, denoted $\Box$, is given by composing the associated 2-cells from $\B$.  The required diagrams commute since they correspond to vertically stacked commutative diagrams. 
\[\xymatrix{P\odot L \ar@{=>}[r]^\delta \ar@{=>}[d]_{\id\odot\phi} & L\odot Q \ar@{=>}[d]^{\phi\odot\id}\\
P\odot M \ar@{=>}[d]_{\id\odot\psi} \ar@{=>}[r]^\beta & M\odot Q \ar@{=>}[d]^{\psi\odot\id}\\
P\odot N \ar@{=>}[r]^\gamma & N\odot Q
}.\]

       \underline{Associativity}: Composition of 2-cells in $\Lambda \B$ is composition in $\B$. 
       Since composition of 2-cells in $\B$ is associative,  composition of morphisms in $\Lambda \B(P, Q)$ is associative. 
        
        \underline{Identity Morphism}: Let $(M, \beta): P \to Q$ be an object in $\Lambda \B(P, Q)$. Then $\id_M$ defines an endomorphism of $(M,\beta)$ since the following diagram commutes.  
        \[\xymatrix{P\odot M \ar@{=>}[d]_{\id_P\odot\id_M} \ar@{=>}[r]^\beta & M\odot Q \ar@{=>}[d]^{\id_M\odot \id_Q}\\
        P\odot M \ar@{=>}[r]_\beta & M\odot Q
        }\] 
        To see that 
        \[\ltwo{\id}_M\Box \ltwo{\psi} = \ltwo{\psi}\quad \text{and}\quad  \ltwo{\psi}\Box \ltwo{\id}_N = \ltwo{\psi}\] 
        for a 2-cell $\ltwo{\psi}: (M, \beta) \Rightarrow (N, \gamma)$ in $\Lambda \B(P, Q)$, 
        it is enough to note that the composition $\Box$ is composition of 2-cells in $\B$.

Since the structure maps are inherited from $\mathcal{B}$ the required diagrams commute. 
\end{proof}

Given objects $(M, \beta)\in \Lambda\B( P,Q)$ and $(N, \gamma)\in \Lambda \B( Q,T)$ in $\Lambda \B$, let 
$\boxmaps{\beta}{\gamma}$ 
be the composite 
% https://q.uiver.app/#q=WzAsMyxbMCwwLCJQXFxvZG90IE1cXG9kb3QgTiJdLFsxLDEsIk0gXFxvZG90IFFcXG9kb3QgTiJdLFsyLDAsIk0gXFxvZG90IE5cXG9kb3QgVCJdLFswLDIsIlxcYm94bWFwc3tcXGJldGF9e1xcZ2FtbWF9IiwwLHsibGV2ZWwiOjJ9XSxbMCwxLCJ7XFxiZXRhXFxvZG90IFxcbWF0aHJte2lkfV9OfSIsMix7ImxldmVsIjoyfV0sWzEsMiwie1xcbWF0aHJte2lkfV9NXFxvZG90IFxcZ2FtbWF9IiwyLHsibGV2ZWwiOjJ9XV0=&macro_url=https%3A%2F%2Fwww.overleaf.com%2Fproject%2F64c9c0fcea71c453e68a663b
\[\begin{tikzcd}
	{P\odot M\odot N} && {M \odot N\odot T} \\
	& {M \odot Q\odot N}
	\arrow["{\boxmaps{\beta}{\gamma}}", Rightarrow, from=1-1, to=1-3]
	\arrow["{{\beta\odot \mathrm{id}_N}}"', Rightarrow, from=1-1, to=2-2]
	\arrow["{{\mathrm{id}_M\odot \gamma}}"', Rightarrow, from=2-2, to=1-3]
\end{tikzcd}\]
Then the  bicategorical composition in $\Lambda \B$ of $(M,\beta)$ and $(N,\gamma)$, denoted  
\[(M, \beta) \boxdot (N, \gamma),\]
is $(M\odot N, \boxmaps{\beta}{\gamma})$. 

\begin{notn}\label{notation_box_dot}
   We will abuse notation and  also use $(M, \beta) \boxdot (N, \gamma)$  to indicate  the  map $\boxmaps{\beta}{\gamma}$.
\end{notn}

\begin{lem}\label{LambdaB_assoc}
 For 1-cells $(L, \delta): P\to Q, (M, \beta): Q\to R$, and $(N, \gamma): R \to S$ in $\Lambda\B$, the associator $a_{L,M,N}\colon (L\odot M)\odot N\to L\odot (M\odot N)$ in $\B$ is a 2-cell 
\[(L, \delta)\boxdot\left[(M, \beta)\boxdot(N, \gamma)\right]
\to (\left[L, \delta)\boxdot(M, \beta)\right]\boxdot(N, \gamma)\] in $\Lambda\B$.
\end{lem}

   In particular, $a_{-,-,-}$ is an associator for $\Lambda\B$ satisfying the pentagon axiom.

\begin{proof}
    We need to show 
the diagram in \eqref{eq:assocator_1-cells} commutes. 
    \begin{equation}\label{eq:assocator_1-cells}
    \xymatrix@C=4cm{P\odot \left[L\odot (M\odot N)\right] \ar@{=>}[r]^{(L, \delta)\boxdot\left[(M, \beta)\boxdot(N, \gamma)\right]} \ar@{<=}[d]_{\id_P\odot a_{L, M, N}} &  \left[L\odot (M\odot N)\right]\odot S \ar@{<=}[d]^{a_{L, M, N}\odot\id_S}\\
    P\odot \left[(L\odot M)\odot N\right] \ar@{=>}[r]_{\left[(L, \delta)\boxdot(M, \beta)\right]\boxdot(N, \gamma)} & \left[(L\odot M)\odot N\right]\odot S}.
    \end{equation}
The top and bottom maps in \eqref{eq:assocator_1-cells} are  the 2-cells 
in \cref{eq:assoctivity_Lambda_B}.  Here we have suppressed the associators  \cite{bicat_coherence,MP:coherence}.  This has the additional benefit that it is easier to see that the diagram in  \eqref{eq:assocator_1-cells} commutes. 
\begin{figure}
     \centering
     \begin{subfigure}[b]{.45\textwidth}\label{lamdba_b_2_1}
% https://q.uiver.app/#q=WzAsNCxbMSwwLCJMXFxvZG90IE1cXG9kb3QgTlxcb2RvdCBTIl0sWzAsMCwiUFxcb2RvdCBMXFxvZG90IE1cXG9kb3QgTiAiXSxbMSwxLCJMXFxvZG90IE1cXG9kb3QgUlxcb2RvdCBOKSAiXSxbMCwxLCJMXFxvZG90IChRXFxvZG90IE0pXFxvZG90IE4iXSxbMSwyLCJ7KEwsXFxkZWx0YSlcXGJveGRvdChNLCBcXGJldGEpXFxvZG90XFxtYXRocm17aWR9X059IiwxLHsibGV2ZWwiOjJ9XSxbMiwwLCJ7XFxtYXRocm17aWR9X3tMXFxvZG90IE19XFxvZG90IFxcZ2FtbWF9IiwyLHsibGV2ZWwiOjJ9XSxbMSwzLCJ7XFxkZWx0YSBcXG9kb3QgXFxtYXRocm17aWR9X01cXG9kb3QgXFxtYXRocm17aWR9X059IiwyLHsibGV2ZWwiOjJ9XSxbMywyLCJ7XFxtYXRocm17aWR9X0xcXG9kb3QgXFxiZXRhXFxvZG90IFxcbWF0aHJte2lkfV9OfSIsMix7ImxldmVsIjoyfV0sWzEsMCwieyhMLCBcXGRlbHRhKVxcYm94ZG90KE0sIFxcYmV0YSlcXGJveGRvdChOLCBcXGdhbW1hKX0iLDAseyJsZXZlbCI6Mn1dXQ==
\adjustbox{scale=.9,center}{\begin{tikzcd}[column sep=1in,row sep=.5in]
	{P\odot L\odot M\odot N } & {L\odot M\odot N\odot S} \\
	{L\odot (Q\odot M)\odot N} & {L\odot M\odot R\odot N) }
	\arrow["{{(L, \delta)\boxdot(M, \beta)\boxdot(N, \gamma)}}", Rightarrow, from=1-1, to=1-2]
	\arrow["{{\delta \odot \mathrm{id}_M\odot \mathrm{id}_N}}"', Rightarrow, from=1-1, to=2-1]
	\arrow["{{(L,\delta)\boxdot(M, \beta)\odot\mathrm{id}_N}}"{description}, Rightarrow, from=1-1, to=2-2]
	\arrow["{{\mathrm{id}_L\odot \beta\odot \mathrm{id}_N}}"', Rightarrow, from=2-1, to=2-2]
	\arrow["{{\mathrm{id}_{L\odot M}\odot \gamma}}"', Rightarrow, from=2-2, to=1-2]
\end{tikzcd}}
\caption{$\left[(L, \delta)\boxdot(M, \beta)\right]\boxdot(N, \gamma)$}
\end{subfigure}
\hfill 
     \begin{subfigure}[b]{0.45\textwidth}\label{lamdba_b_1_2}
% https://q.uiver.app/#q=WzAsNCxbMSwwLCJMXFxvZG90IE1cXG9kb3QgTlxcb2RvdCBTIl0sWzAsMCwiIFBcXG9kb3QgTCBcXG9kb3QgTVxcb2RvdCBOIl0sWzAsMSwiTFxcb2RvdCBRIFxcb2RvdCBNXFxvZG90IE4gIl0sWzEsMSwiTFxcb2RvdCBNXFxvZG90IFJcXG9kb3QgTiAiXSxbMSwyLCJ7XFxkZWx0YVxcb2RvdCBcXG1hdGhybXtpZH1fe01cXG9kb3QgTn19IiwyLHsibGV2ZWwiOjJ9XSxbMSwwLCJ7KEwsIFxcZGVsdGEpXFxib3hkb3QoTSwgXFxiZXRhKVxcYm94ZG90KE4sIFxcZ2FtbWEpfSIsMCx7ImxldmVsIjoyfV0sWzIsMywie1xcbWF0aHJte2lkfV9MXFxvZG90IFxcYmV0YVxcb2RvdCBcXG1hdGhybXtpZH1fTn0iLDIseyJsZXZlbCI6Mn1dLFsyLDAsIntcXG1hdGhybXtpZH1fTFxcb2RvdCAoKE0sIFxcYmV0YSlcXGJveGRvdChOLCBcXGdhbW1hKSl9IiwxLHsibGV2ZWwiOjJ9XSxbMywwLCJ7XFxtYXRocm17aWR9X0xcXG9kb3QgXFxtYXRocm17aWR9X01cXG9kb3QgXFxnYW1tYX0iLDIseyJsZXZlbCI6Mn1dXQ==
\adjustbox{scale=.9,center}{\begin{tikzcd}[column sep=1in,row sep=.5in]
	{ P\odot L \odot M\odot N} & {L\odot M\odot N\odot S} \\
	{L\odot Q \odot M\odot N } & {L\odot M\odot R\odot N }
	\arrow["{{(L, \delta)\boxdot(M, \beta)\boxdot(N, \gamma)}}", Rightarrow, from=1-1, to=1-2]
	\arrow["{{\delta\odot \mathrm{id}_{M\odot N}}}"', Rightarrow, from=1-1, to=2-1]
	\arrow["{{\mathrm{id}_L\odot ((M, \beta)\boxdot(N, \gamma))}}"{description}, Rightarrow, from=2-1, to=1-2]
	\arrow["{{\mathrm{id}_L\odot \beta\odot \mathrm{id}_N}}"', Rightarrow, from=2-1, to=2-2]
	\arrow["{{\mathrm{id}_L\odot \mathrm{id}_M\odot \gamma}}"', Rightarrow, from=2-2, to=1-2]
\end{tikzcd}}
\caption{$(L, \delta)\boxdot\left[(M, \beta)\boxdot(N, \gamma)\right]$}
\end{subfigure}
\caption{Compositions of three 1-cells in $\Lambda \B$}\label{eq:assoctivity_Lambda_B}
\end{figure}
\end{proof}

\begin{rmk}  In most of the following diagrams we will suppress structure morphisms as we did in \cref{eq:assoctivity_Lambda_B}.  This is justified by the coherence results in \cite{bicat_coherence,MP:coherence}.   We will retain them in examples like \eqref{eq:assocator_1-cells} where they are essential to understanding the diagram.
\end{rmk}

Let $P: V\to V$ be an object in $\Lambda \B$. Let $(U_V, r\ell^{-1}): P\to P$  be the 1-cell in $\Lambda \B$ equipped with the composition \[\xymatrix{P\odot U_V \ar@{=>}[r]^-r & P \ar@{=>}[r]^-{\ell^{-1}} & U_V \odot P,} \] where $U_V$ is the unit object for $V$ in $\B$ and $r, \ell$ are the right and left unitors in $\B$, respectively. 

\begin{lem}\label{LambdaB_unit} Let  $N: V\to R$ be a 1-cell in $\B$ and $(N, \gamma): P\to Q$ be a 1-cell in $\Lambda \B$.
The unitor maps 
\[r\colon N\odot U_R\to N\text{ and }\ell\colon U_V\odot  N\to N\] 
in $\B$ are 2-cells 
\[(N, \gamma)\boxdot (U_R, \ell^{-1}r) \Rightarrow (N, \gamma)\text{ and } (U_V, r^{-1}\ell)\boxdot (N, \gamma)\Rightarrow (N, \gamma)\] in $\Lambda \B$. 
\end{lem}
In particular, the unitors in $\B$ are unitors in $\Lambda\B$ satisfying the triangle axiom.

\begin{proof}
The commutative diagram \cref{existremunitor} shows the map $r$ is a 2-cell in $\Lambda \B$,
$\ltwo{r}: (N, \gamma)\boxdot (U_R, r\ell^{-1}) \Rightarrow (N, \gamma).$  The case of $\ell$ is similar.  
\end{proof}
\begin{figure}
\[\begin{tikzcd}
	{P\odot (N\odot U_R)} && {P\odot N} \\
	{ (P\odot N)\odot U_R} && {P\odot N} \\
	{ (N\odot Q)\odot U_R} && {N\odot Q} \\
	{ N\odot (Q\odot U_R)} && {N\odot Q} \\
	{ N\odot (U_R\odot Q)} && {N\odot Q} \\
	{ (N\odot U_R)\odot Q} && {N\odot Q}
	\arrow["{\id\odot r}", Rightarrow, from=1-1, to=1-3]
	\arrow["{\text{Naturality}}"{description}, draw=none, from=1-1, to=2-3]
	\arrow[equals, from=1-3, to=2-3]
	\arrow["a", Rightarrow, from=2-1, to=1-1]
	\arrow["r", Rightarrow, from=2-1, to=2-3]
	\arrow["{\gamma\odot\id}"', Rightarrow, from=2-1, to=3-1]
	\arrow["\gamma", Rightarrow, from=2-3, to=3-3]
	\arrow["{\text{Nat. of Unitor}}"{description}, draw=none, from=3-1, to=2-3]
	\arrow["r", Rightarrow, from=3-1, to=3-3]
	\arrow["a"', Rightarrow, from=3-1, to=4-1]
	\arrow[equals, from=3-3, to=4-3]
	\arrow["{\text{Triangle Axiom}}"{description}, draw=none, from=4-1, to=3-3]
	\arrow["{\id\odot r}", Rightarrow, from=4-1, to=4-3]
	\arrow["{\id\odot r\ell^{-1}}"', Rightarrow, from=4-1, to=5-1]
	\arrow[equals, from=4-3, to=5-3]
	\arrow["{=}"{description}, no body, from=5-1, to=4-3]
	\arrow["{\id\odot\ell}", Rightarrow, from=5-1, to=5-3]
	\arrow["{\text{Triangle Axiom}}"{description}, draw=none, from=5-3, to=6-1]
	\arrow[equals, from=5-3, to=6-3]
	\arrow["a", Rightarrow, from=6-1, to=5-1]
	\arrow["{r\odot\id}", Rightarrow, from=6-1, to=6-3]
\end{tikzcd}\]
\caption{Existence of unitor for $\Lambda \B$}\label{existremunitor}
\end{figure}

 \cref{LambdaBPQ_Cat,LambdaB_assoc,LambdaB_unit} complete the proof of \cref{Lambda_bicat} that $\Lambda\B$ is a bicategory. 

\section{Functors induced on twisted endomorphism bicategories}\label{sec:remnant_functors} 
In this section we verify that the construction of \cref{sec:remnant} extends to functors. 

\begin{lem}\label{inducted_functor_lem_1} The 2-cell $\phi \colon F(M)\odot F(N)\to F(M\odot N)$ in $\B$ is a 2-cell
\[\Lambda F(M,\beta)\boxdot \Lambda F(N,\gamma)\to \Lambda F((M,\beta)\boxdot (N,\gamma))\]
 in $\Lambda \B$.

\end{lem}

\begin{proof}

\begin{figure}
\[\begin{tikzcd}[column sep= .0in, row sep =. 35in]
	{F(P)\odot F(M)\odot F(N)} &&& {F(P)\odot F(M\odot N)} \\
	& {F(P\odot M)\odot F(N)} & {F(P\odot M\odot N)} \\
	& {F(M\odot Q)\odot F(N)} \\
	& {F(M)\odot F(Q)\odot F(N)} & {F(M\odot Q\odot N)} \\
	& {F(M)\odot F(Q\odot N)} \\
	& {F(M)\odot F(N\odot T)} & {F(M\odot N\odot T)} \\
	{F(M)\odot F(N)\odot F(T)} &&& {F(M\odot N)\odot F(T)}
	\arrow["{\id \odot \phi}", Rightarrow,  from=1-1, to=1-4]
	\arrow["{\phi\odot \mathrm{id}}"', Rightarrow,  from=1-1, to=2-2]
	\arrow["{\Lambda F(M,\beta)\boxdot \Lambda F(N,\gamma)}"', Rightarrow,  dotted, from=1-1, to=7-1]
	\arrow["\phi", Rightarrow,  from=1-4, to=2-3]
	\arrow["{\Lambda F((M,\beta)\boxdot(N,\gamma))}", Rightarrow,  dotted, from=1-4, to=7-4]
	\arrow["\phi", Rightarrow,  from=2-2, to=2-3]
	\arrow["{F(\beta)\odot \mathrm{id}}", Rightarrow,  from=2-2, to=3-2]
	\arrow["{F(\beta\odot \mathrm{id})}", Rightarrow,  from=2-3, to=4-3]
	\arrow["\phi", Rightarrow,  from=3-2, to=4-3]
	\arrow["{\phi\odot \mathrm{id}}"', Rightarrow,  from=4-2, to=3-2]
	\arrow["{\mathrm{id}\odot \phi}", Rightarrow,  from=4-2, to=5-2]
	\arrow["{F(\mathrm{id}\odot \gamma)}", Rightarrow,  from=4-3, to=6-3]
	\arrow["\phi"', Rightarrow,  from=5-2, to=4-3]
	\arrow["{\mathrm{id}\odot F(\gamma)}", Rightarrow,  from=5-2, to=6-2]
	\arrow["\phi"', Rightarrow,  from=6-2, to=6-3]
	\arrow["{\mathrm{id}\odot \phi}", Rightarrow,  from=7-1, to=6-2]
	\arrow["{\phi\odot \id}"', Rightarrow,  from=7-1, to=7-4]
	\arrow["\phi"', Rightarrow,  from=7-4, to=6-3]
\end{tikzcd}\]
\caption{Compatibility of $\Lambda F$ and $\boxdot$}\label{Lambda_F_boxdot}
\end{figure}
We show that $\phi \colon F(M)\odot F(N)\to F(M\odot N)$ defines a 2-cell in $\Lambda \B$ by showing that the diagram in \eqref{boxdot_Lambda} commutes.
\begin{equation}\label{boxdot_Lambda}
% https://q.uiver.app/#q=WzAsNCxbMCwxLCJGKFApXFxvZG90IEYoTVxcb2RvdCBOKSJdLFszLDEsIkYoTVxcb2RvdCBOKVxcb2RvdCBGKFQpIl0sWzAsMCwiRihQKVxcb2RvdCAoRihNKVxcb2RvdCBGKE4pKSJdLFszLDAsIihGKE0pXFxvZG90IEYoTikpXFxvZG90IEYoVCkiXSxbMCwxLCJcXExhbWJkYSBGKChNLFxcYmV0YSlcXGJveGRvdCAoTixcXGdhbW1hKSkiLDAseyJsZXZlbCI6Miwic3R5bGUiOnsiYm9keSI6eyJuYW1lIjoiZG90dGVkIn19fV0sWzIsMywiXFxMYW1iZGEgRihNLFxcYmV0YSlcXGJveGRvdCBcXExhbWJkYSBGKE4sXFxnYW1tYSkiLDAseyJsZXZlbCI6Miwic3R5bGUiOnsiYm9keSI6eyJuYW1lIjoiZG90dGVkIn19fV0sWzIsMCwiXFxpZFxcb2RvdFxccGhpIiwyLHsibGV2ZWwiOjJ9XSxbMywxLCJcXHBoaVxcb2RvdFxcaWQiLDAseyJsZXZlbCI6Mn1dXQ==&macro_url=https%3A%2F%2Fwww.overleaf.com%2Fproject%2F64c9c0fcea71c453e68a663b
\begin{tikzcd}
	{F(P)\odot (F(M)\odot F(N))} &&& {(F(M)\odot F(N))\odot F(T)} \\
	{F(P)\odot F(M\odot N)} &&& {F(M\odot N)\odot F(T)}
	\arrow["{\Lambda F(M,\beta)\boxdot \Lambda F(N,\gamma)}", Rightarrow, dotted, from=1-1, to=1-4]
	\arrow["{\id\odot\phi}"', Rightarrow, from=1-1, to=2-1]
	\arrow["{\phi\odot\id}", Rightarrow, from=1-4, to=2-4]
	\arrow["{\Lambda F((M,\beta)\boxdot (N,\gamma))}", Rightarrow, dotted, from=2-1, to=2-4]
\end{tikzcd}
\end{equation}
See \cref{Lambda_F_boxdot}.
\end{proof}

\begin{lem}\label{inducted_functor_lem_2}
 For a 0-cell $V$ in $\Lambda \B$, the 2-cell $\iota\colon U_{F(V)}\to F(U_V)$ is a 2-cell $(U_{F(V)},r\ell^{-1})\to \Lambda F(U_V,r\ell^{-1})$.
\end{lem}

\begin{proof}
The image of $(U_V,r\ell^{-1})$ under $\Lambda F$ is the pair 
\[(F(U_V), F(P)\odot F(U_V)\Rightarrow F(P\odot U_V)
\Rightarrow F(P) \Rightarrow F(U_V\odot P)\Rightarrow F(U_V)\odot F(P)).\]
Since the diagram in \cref{lambda_F_unit} commutes, 
$\iota$ defines a map in $\Lambda \cc$.
\begin{figure}
\centering
% https://q.uiver.app/#q=WzAsOCxbMSwwLCJGKFApXFxvZG90IEYoVV9WKSJdLFsxLDEsIkYoUFxcb2RvdCBVX1YpIl0sWzEsMiwiRihQKSJdLFsxLDMsIkYoVV9WXFxvZG90IFApIl0sWzEsNCwiRihVX1YpXFxvZG90IEYoUCkiXSxbMCwwLCJGKFApXFxvZG90IFVfe0YoVil9Il0sWzAsNCwiVV97RihWKX1cXG9kb3QgRihQKSJdLFswLDIsIkYoUCkiXSxbNywyLCIiLDIseyJsZXZlbCI6Miwic3R5bGUiOnsiaGVhZCI6eyJuYW1lIjoibm9uZSJ9fX1dLFs1LDcsInIiLDIseyJsZXZlbCI6Mn1dLFs3LDYsIlxcZWxsXnstMX0iLDIseyJsZXZlbCI6Mn1dLFs1LDAsIlxcbWF0aHJte2lkfVxcb2RvdCBcXGlvdGEiLDIseyJsZXZlbCI6Mn1dLFswLDEsIlxccGhpIiwwLHsibGV2ZWwiOjJ9XSxbMSwyLCJGKHIpIiwwLHsibGV2ZWwiOjJ9XSxbMiwzLCJGKFxcZWxsXnstMX0pIiwwLHsibGV2ZWwiOjJ9XSxbMyw0LCJcXHBoaV57LTF9IiwwLHsibGV2ZWwiOjJ9XSxbNiw0LCJcXGlvdGFcXG9kb3QgXFxtYXRocm17aWR9IiwyLHsibGV2ZWwiOjJ9XV0=
\begin{tikzcd}
	{F(P)\odot U_{F(V)}} & {F(P)\odot F(U_V)} \\
	& {F(P\odot U_V)} \\
	{F(P)} & {F(P)} \\
	& {F(U_V\odot P)} \\
	{U_{F(V)}\odot F(P)} & {F(U_V)\odot F(P)}
	\arrow[Rightarrow, no head, from=3-1, to=3-2]
	\arrow["r"', Rightarrow, from=1-1, to=3-1]
	\arrow["{\ell^{-1}}"', Rightarrow, from=3-1, to=5-1]
	\arrow["{\mathrm{id}\odot \iota}"', Rightarrow, from=1-1, to=1-2]
	\arrow["\phi", Rightarrow, from=1-2, to=2-2]
	\arrow["{F(r)}", Rightarrow, from=2-2, to=3-2]
	\arrow["{F(\ell^{-1})}", Rightarrow, from=3-2, to=4-2]
	\arrow["{\phi^{-1}}", Rightarrow, from=4-2, to=5-2]
	\arrow["{\iota\odot \mathrm{id}}"', Rightarrow, from=5-1, to=5-2]
\end{tikzcd}
\caption{$\Lambda F$ respects the unit}\label{lambda_F_unit}
\end{figure}
\end{proof}

\lmabdaofacomposite*

\begin{proof}
   If $F$ is a strong functor, $F(M)$ is dualizable.   Then the pair 
    \begin{equation}\label{1_cell_Lambda_F}(F(M), \phi\circ F(\psi)\circ \phi^{-1}).
    \end{equation} 
 is a 1-cell in $\Lambda \mathcal{C}$. 
By \cref{inducted_functor_lem_1,inducted_functor_lem_2},
the coherence diagrams follow from the coherence diagrams for $F$.

    The compatibility with composition is straightforward to check. 
    On 0- and 2-cells, $\Lambda F=F$ and there is nothing to prove.  For 1-cells, 
    \begin{align*}
        (\Lambda (F\circ G))(M,\beta)
        &=((F\circ G)(M), \phi_{F\circ G} \circ ((F\circ G)(\beta)) \circ \phi_{F\circ G}^{-1})\\
        &=((F\circ G)(M), \phi_{F}\circ F(\phi_{G})\circ ((F\circ G)(\beta)) \circ F(\phi_{G}^{-1})\circ \phi_F^{-1})\\
        &=(\Lambda F)(G(M), \phi_{G}\circ G(\beta) \circ \phi_{G}^{-1})
        \\
        &=(\Lambda F)(\Lambda G)(M,\beta)
    \end{align*}
\end{proof}

\section{Traces on twisted endomorphism bicategories and characters}\label{sec:traces}
In this section we show  the bicategorical trace can be upgraded to a functor on the endomorphism bicategory of \cref{sec:remnant} and use this functor to define characters for 2-representations.
\thmcharacter*

To prove this theorem we first note the following consequence of  \cite[Proposition 7.5]{Ponto_2012}.

\begin{lem} \label{lem:comp}
    For the functor  $\rchi$ defined above and composable 1-cells $(M, \beta)$ and $(N, \gamma)$ in $\Lambda \B$, 
    \[\rchi[(N, \gamma)\boxdot (M, \beta)] = \rchi[(N, \gamma)]\circ \rchi[(M, \beta)].\] 
\end{lem}

We now verify that $\chi$ preserves unit 1-cells.

 \begin{lem} \label{lem:unit}
     For 
     the unit 1-cell $(U_V, r\ell^{-1}): P\to P$ in $\Lambda \B$, 
     \[\rchi[(U_V, r\ell^{-1})] = \id_{\sh{P}}.\]
 \end{lem}

 \begin{proof}
Recall that $(U_R, U_R)$ forms a dual pair, with coevaluation 
$\eta = \ell^{-1}=r^{-1}$ and evaluation 
$\epsilon = \ell=r$. 

In the diagram in \cref{fig:lambda_functor_unit}, $\rchi[(U_V, r\ell^{-1})]$ is the  top, right, and bottom composite.  
\begin{figure}
\begin{tikzcd}
	{\sh{P}} & {\sh{P\odot U_R}} && {\sh{P\odot(U_R\odot U_R)}} \\
	&&& {\sh{(P\odot U_R)\odot U_R}} \\
	& {\sh{P\odot U_R}} && {\sh{P\odot U_R}} \\
	& {\sh{P\odot U_R}} && {\sh{(U_R\odot P)\odot U_R}} \\
	& {\sh{U_R\odot P}} && {\sh{U_R\odot (U_R\odot P)}} \\
	{\sh{P}} & {\sh{U_R\odot P}} && {\sh{(U_R\odot U_R)\odot P}}
	\arrow["{\sh{r^{-1}}}", Rightarrow, from=1-1, to=1-2]
	\arrow["\id", equals, from=1-1, to=6-1]
	\arrow["{\sh{\id\odot \eta}}", curve={height=-24pt}, Rightarrow, from=1-2, to=1-4]
	\arrow["{\sh{\id\odot \ell^{-1}}}"', Rightarrow, from=1-2, to=1-4]
	\arrow["\id"', equals, from=1-2, to=3-2]
	\arrow["{\sh{a^{-1}}}"', Rightarrow, from=2-4, to=1-4]
	\arrow["{\sh{r\odot\id}}", Rightarrow, from=2-4, to=3-4]
	\arrow["\id"', equals, from=3-2, to=3-4]
	\arrow["\id", equals, from=3-2, to=4-2]
	\arrow["{\sh{\ell^{-1}\odot\id}}", Rightarrow, from=3-4, to=4-4]
	\arrow["\theta"', Rightarrow, from=4-2, to=5-2]
	\arrow["{\sh{\ell\odot \id}}"', Rightarrow, from=4-4, to=4-2]
	\arrow["\theta", Rightarrow, from=4-4, to=5-4]
	\arrow["\id"', equals, from=5-2, to=6-2]
	\arrow["{\sh{\id\odot \ell}}"', Rightarrow, from=5-4, to=5-2]
	\arrow["{\sh{\ell}}"', Rightarrow, from=6-2, to=6-1]
	\arrow["{\sh{a^{-1}}}"', Rightarrow, from=6-4, to=5-4]
	\arrow["{\sh{\epsilon\odot\id}}", curve={height=-24pt}, Rightarrow, from=6-4, to=6-2]
	\arrow["{\sh{r\odot\id}}"', Rightarrow, from=6-4, to=6-2]
\end{tikzcd}
\caption{ $\rchi[(U_V, r\ell^{-1})]$ }\label{fig:lambda_functor_unit}
\end{figure}
The regions of the diagram commute by coherence for bicategories and shadows \cite{MP:coherence}.
 \end{proof}

To consider the behavior on 2-cells we need a preliminary lemma. 
\begin{lem} \label{lem:treq}
    Let $\B$ be a bicategory with shadow functors $\sh{\text{-}}$. 
    Given a dualizable 1-cell $M$, isomorphism 2-cell $\phi\colon M\Rightarrow N$, and commuting diagram 
    % https://q.uiver.app/#q=WzAsNCxbMCwwLCJQXFxvZG90IE0iXSxbMSwwLCJNXFxvZG90IFEiXSxbMCwxLCJQXFxvZG90IE4iXSxbMSwxLCJOXFxvZG90IFEiXSxbMCwyLCJcXGlkIFxcb2RvdCBcXHBoaSIsMix7ImxldmVsIjoyfV0sWzAsMSwiXFxiZXRhIiwwLHsibGV2ZWwiOjJ9XSxbMiwzLCJcXGdhbW1hIiwwLHsibGV2ZWwiOjJ9XSxbMSwzLCJcXHBoaVxcb2RvdCBcXGlkIiwwLHsibGV2ZWwiOjJ9XV0=&macro_url=https%3A%2F%2Fwww.overleaf.com%2Fproject%2F64c9c0fcea71c453e68a663b
\[\begin{tikzcd}
	{P\odot M} & {M\odot Q} \\
	{P\odot N} & {N\odot Q}
	\arrow["\beta", Rightarrow, from=1-1, to=1-2]
	\arrow["{\id \odot \phi}"', Rightarrow, from=1-1, to=2-1]
	\arrow["{\phi\odot \id}", Rightarrow, from=1-2, to=2-2]
	\arrow["\gamma", Rightarrow, from=2-1, to=2-2]
\end{tikzcd}\]
     then 
   $\tr(\beta) = \tr(\gamma).$
\end{lem}

\begin{proof}
       Let  $\eta_M$ and $\epsilon_M$ be a choice of coevaluation and evaluation for $M$ with dual $M^\ast$.
    Since $N$ is isomorphic to $M$ and $M$ is dualizable, $N$ is dualizable and a choice of coevaluation and evaluation for $N$ are 
    \[\xymatrix{\eta: U_R \ar@{=>}[r]^-{\eta_M} & M\odot M^\ast \ar@{=>}[r]^{\phi\odot\id} & N\odot M^\ast}\] and
    \[\xymatrix{\epsilon: M^\ast \odot N \ar@{=>}[r]^-{\id\odot\phi^{-1}} & M^\ast\odot M \ar@{=>}[r]^-{\epsilon_M} & U_S}.\]

    Since traces are independent of the choice of dual \cite[4.5.2]{p:thesis}, we can use this dual pair to compute the trace of $\gamma$. In the commuting diagram in \cref{comparison_of_traces}, the left composite is the trace of $\beta$ and the right composite is the trace of $\gamma$.
\begin{figure}
\centering
\begin{tikzcd}
	{\sh{P}} \\
	{\sh{P\odot U_S} } & {\sh{P\odot U_S} } \\
	{ \sh{P\odot M\odot M^\ast}} && {\sh{P\odot N\odot M^\ast}} \\
	{\sh{M\odot Q\odot M^\ast}} && {\sh{N\odot Q\odot M^\ast}} \\
	{\sh{Q\odot M^\ast\odot M} } && {\sh{Q\odot M^\ast\odot N}} \\
	{\sh{Q \odot U_S}} & {\sh{Q \odot U_S}} \\
	{\sh{Q}}
	\arrow["{\sh{r^{-1}}}"', Rightarrow, from=1-1, to=2-1]
	\arrow["{\sh{r^{-1}}}", Rightarrow, from=1-1, to=2-2]
	\arrow["{\sh{\id\odot\eta_M}}"', Rightarrow, from=2-1, to=3-1]
	\arrow[Rightarrow, no head, from=2-1, to=2-2]
	\arrow["{\sh{\beta\odot\id}}"', Rightarrow, from=3-1, to=4-1]
	\arrow["{\sh{\id\odot\phi\odot\id}}", Rightarrow, from=3-1, to=3-3]
	\arrow["\theta"', Rightarrow, from=4-1, to=5-1]
	\arrow["{\sh{\phi\odot\id\odot\id}}", Rightarrow, from=4-1, to=4-3]
	\arrow["{\sh{\id\odot\epsilon_M}}"', Rightarrow, from=5-1, to=6-1]
	\arrow["{\sh{\id\odot\id\odot\phi}}", Rightarrow, from=5-1, to=5-3]
	\arrow["{\sh{r}}"', Rightarrow, from=6-1, to=7-1]
	\arrow[Rightarrow, no head, from=6-1, to=6-2]
	\arrow["{\sh{\id\odot\eta}}", Rightarrow, dotted, from=2-2, to=3-3]
	\arrow["{\sh{\gamma\odot\id}}", Rightarrow, from=3-3, to=4-3]
	\arrow["\theta", Rightarrow, from=4-3, to=5-3]
	\arrow["{\sh{\id\odot\epsilon}}", Rightarrow, dotted, from=5-3, to=6-2]
	\arrow["{\sh{r}}", Rightarrow, from=6-2, to=7-1]
\end{tikzcd}
\caption{Comparison of traces (\cref{lem:treq})}\label{comparison_of_traces}
\end{figure}
\end{proof}

Composition of 2-cells is preserved by the map $\rchi$ follows from the limited number of 2-cells in the target category.

 \begin{lem} \label{lem:2comp}
     For the map $\rchi$ defined above and composable 2-cells $\ltwo{\phi}: (M, \beta) \Rightarrow (N, \gamma)$ and $\ltwo{\psi}: (N, \gamma) \Rightarrow (L, \sigma)$, 
     \[\rchi(\ltwo{\phi}\Box \ltwo{\psi}) = \rchi(\ltwo{\phi})\circ \rchi(\ltwo{\psi}).\]
 \end{lem}

\begin{proof}
    $\rchi(\ltwo{\phi})$ is a 2-cell  $\tr(\beta)\Rightarrow \tr(\gamma)$. The only 2-cells in $T$ are identities. By \cref{lem:treq} and $M\cong N\cong L$, $\tr(\beta)=\tr(\gamma)=\tr(\sigma)$, which means $\id_{\tr(\beta)} = \id_{\tr(\gamma)} = \id_{\tr(\sigma)}$. 
\end{proof}

\noindent
\begin{proof}[Proof of \cref{thm:character}]
    \cref{lem:comp} shows that $\rchi$ respects 1-cell composition. \cref{lem:unit} shows that $\rchi$ preserves the unit 1-cell. \cref{lem:2comp} shows that $\rchi$ respects 2-cell composition. $\rchi$ maps 2-cells in $\Lambda\B$ to identities, so $\rchi$ preserves the unit 2-cell trivially. 
\end{proof}
We now use the functor in \cref{thm:character} to define the character of a strong 2-functor (and so a 2-representation).

\defcharacter*

Conjugation invariance is fundamental to this definition.  In the case that $\mathcal{C}=BG$, the 1-cells in $\Lambda BG$ are equalities \[gh=hk\] for elements $g$, $h$, and $k$ of $G$, and 
traces associated to these equalities are the fundamental objects of interest.  For other choices of $\mathcal{C}$, the morphisms of $\Lambda \mathcal{C}$ are 2-cells 
\[P\odot M\to M\odot Q\] which we consider to be generalized conjugation relations between 1-cells. 

\section{Coshadows, cotraces, and cotwisted endomorphism bicategories}\label{sec:coshadows}
This section has two goals.  The first is to show that the approach above to characters in terms of the bicategorical trace extends to the bicategorical cotrace \cite{barhite2023bicategorical}.  The second goal, and the reason for the relevance of the first, is to more clearly illuminate how the approach of \cite{GK} sits inside the approach here.

Fundamental to the Bartlett and  Ganter--Kapranov character is the following definition.
\begin{defn}\cite[4.4]{Bartlett}\cite[3.1]{GK} Let $F\colon x\to x$ be an endo-1-cell in a bicategory $\mathcal{B}$. Then the {\bf categorical trace} of $F$ is \[\Tr(F)= 2\Hom_{\mathcal{B}}(U_x, F).\] 
\end{defn}
That is, the categorical trace of $F$ is the set of 2-cells from the identity 1-cell on $x$ to $F$.   

\begin{ex} The case of $2\Vect_k$ is particularly illuminating.  For a 1-cell $A = [A_{ij}]$ in $2\Vect_k$, 
\[\Tr(A) = 2\Hom_{2\Vect_k}(1_{[n]}, A) \cong \overset{n}{\underset{i=1}{\bigoplus}}\ A_{ii}.\]
\end{ex}

The categorical trace is not a shadow in the sense of \cref{defn:shadow}, but it a coshadow as defined by Barhite in \cite{barhite2023bicategorical}.   (See \cite[Example 4.4]{barhite2023bicategorical}.)
In this section we briefly describe the generalization of the approach above to characters using shadows to coshadows.  The arguments are the same, but this gives an approach in closer alignment to that in \cite{GK}.

Coshadows have the same properties as  shadows but with respect to the internal hom rather than the bicategorical composition.  A closed 
$\B$ bicategory  is equipped with left and right internal hom-functors
\[
	- \triangleleft - : \B(R, T) \times \B(R, S)^{op} \to \B(S, T)
\]
and
\[
	- \triangleright - : \mathscr{B}(S, T)^{op} \times \mathscr{B}(R, T) \to \mathscr{B}(R, S)
\]
for all triples of 0-cells $R, S, T$ and natural isomorphisms
\begin{equation}
\label{eq:closed_bicat_tensor_hom}
	\mathscr{B}(S, T)(N, P \triangleleft M) \cong \mathscr{B}(R, T)(M \odot N, P) \cong \mathscr{B}(R, S)(M, N \triangleright P)
\end{equation}
for all triples of 1-cells $M : R \to S$, $N : S \to T$, and $P : R \to T$.

\begin{defn}[Compare to \cref{defn:shadow}]\cite[Def 4.1]{barhite2023bicategorical}
    A \textbf{coshadow} for a closed bicategory $\B$ is a category $\mathbf{T}$ and functors \[\coshh{\text{-}}: \B(R, R) \to \mathbf{T} \] for each 0-cell $R$ of $\B$, equipped with natural isomorphisms \[\theta: \coshh{M\triangleright N} \overset{\cong}{\to} \coshh{N \triangleleft M}\] for each $M, N \in \B(R, S)$, such that the following diagrams commute whenever they make sense:

    \[\xymatrix{\coshh{(M\odot N)\triangleright P} \ar[r]^\theta_\cong \ar[d]^\cong_\theta & \coshh{P\triangleleft (M\odot N)} \ar[r]^{\coshh{t}}_\cong & \coshh{(P\triangleleft M)\triangleleft N} \ar[d]^-\theta_\cong\\
    \coshh{M\triangleright (N\triangleright P)} \ar[r]^\cong_\theta & \coshh{(N\triangleright P)\triangleleft M} \ar[r]^\cong_{\coshh{a}} & \coshh{N\triangleright (P\triangleleft M)}}\]

    \[\xymatrix{\coshh{U_R\triangleright M} \ar[r]^\theta_\cong \ar[dr]^\cong_{\coshh{\overline{r}^{-1}}} & \coshh{M\triangleleft U_R} \ar[r]^\theta_\cong \ar[d]^\cong_{\coshh{\overline{\ell}^{-1}}} & \coshh{U_R\triangleright M} \ar[dl]^{\coshh{\overline{r}^{-1}}}_\cong \\ 
    & \coshh{M} &}\]
\end{defn}

The definition of the cotrace parallels the definition of the trace in \cref{defn:bicat_trace}.

\begin{defn}\cite[Definition 4.5]{barhite2023bicategorical}
    Let $\B$ be a closed bicategory with a coshadow and $(M, M^\ast)$ a dual pair with $M\in \B(R, S)$. The \textbf{cotrace} of a 2-cell $f: M\triangleright Q \to P\triangleleft M$, denoted $\co(f)$, is the composite \small{\begin{align*}\coshh{Q} \underset{\cong}{\xto{\overline{r}}} \coshh{U_S \triangleright Q} \underset{\cong}{\xto{\coshh{\epsilon\triangleright\id}}} \coshh{(M^\ast\odot M)\triangleright Q} \underset{\cong}{\xto{\coshh{t}}} \coshh{M^\ast \triangleright (M\triangleright Q)}
    \underset{\cong}{\xto{\coshh{\id\triangleright f}}} \coshh{M^\ast\triangleright(P\triangleleft M)}\\ \hfill\underset{\cong}{\xto{\theta}} \coshh{(P\triangleleft M)\triangleleft M^\ast} \underset{\cong}{\xto{\coshh{t_\ast^{-1}}}} \coshh{P\triangleleft (M\odot M^\ast)}
    \hfill \underset{\cong}{\xto{\id\triangleleft\eta}} \coshh{P\triangleleft U_R} \underset{\cong}{\xto{\coshh{\overline{\ell}^{-1}}}} \coshh{P}\end{align*}}
\end{defn}

Not only are the definitions of the trace and cotrace parallel, but formal results for traces have almost identical proofs for cotraces.  In particular, the results in \cref{sec:remnant,sec:remnant_functors,sec:traces} have the following companions.  (We omit the proofs since they follow so closely.)

Paralleling \cref{Lambda_bicat} we have the following.

\begin{thm}\label{Lambda_bicat_coshadow}\label{coLambdaB_bicat}
    From a closed  bicategory $\mathcal{B}$ there is a bicategory $\cb$ where 
    \begin{itemize}
        \item the 0-cells are the endomorphism 1-cells of $\mathcal{B}$,
        \item the 1-cells $P\to Q$ are pairs $(M,\bar{\alpha})$ where $M$ is a dualizable 1-cell in $\mathcal{B}$ and 
        $\bar{\alpha}$ is a 2-cell 
        \[M\triangleright  P\Rightarrow Q\triangleleft M,\]
        and 
        \item the 2-cells $(M,\bar{\alpha})\Rightarrow (M',\bar{\alpha}')$ are isomorphism 2-cells $\psi\colon M\to M$ of $\mathcal{B}$ so that the diagram 
          \[\xymatrix{M\triangleright P\ar@{=>}[d]_{\psi\triangleright \id } \ar@{=>}[r]^{\bar{\alpha}} 
          & Q\triangleleft M \ar@{=>}[d]^{\id\triangleleft\psi}\\
    M'\triangleright P \ar@{=>}[r]_{\bar{\alpha}'} & Q\triangleleft M'}\]
    commutes.
    \end{itemize}
\end{thm}
We call this bicategory the \textbf{\coremant\ bicategory}.

\begin{thm}[Compare to \cref{lambda_of_a_composite}]\label{coLambdaF}\label{colambda_of_a_composite}    A strong functor $F\colon \B\to \mathcal{C}$ induces a strong functor $\cob{F}\colon \cb\to \cob{\mathcal{C}}$ of \coremant\, bicategories.
    If $F\colon \B\to \mathcal{C}$ and $G\colon \mathcal{C}\to \mathcal{D}$ are strong functors of bicategories, 
\[\cob{(F\circ G)}=(\cob{F})\circ (\cob{G}).\]
\end{thm}

\begin{thm}[Compare to \cref{thm:character}]
Let $\B$ be a closed bicategory with coshadow functors $\coshh{\text{-}}: \B(R, R) \to \mathbf{T}$. Let $\T$ be $\mathbf{T}$ regarded as a 2-category with identity 2-cells. Define $\overline{\rchi}: \cb \to \T$ by \begin{gather*} \overline{\rchi}(P) = \coshh{P},\\ \overline{\rchi}([N, \overline{\gamma}]) = \co(\overline{\gamma}),\ \text{and} \\ \overline{\rchi}(\lone{\psi}) = \id_{\co(\overline{\gamma})}, \end{gather*} where $\co$ is the bicategorical cotrace with coshadow $\coshh{\text{-}}$. Then $\overline{\rchi}$ is a strict 2-functor.
\end{thm}

 \begin{defn}[Compare to \cref{lab:def:character}]\label{cocharacter}
    Let $\B$ be a closed bicategory with coshadow $\coshh{}$ valued in $\mathbf{T}$ and $\rho\colon BG\to \B$ be a 2-representation. Let $\T$ be $\mathbf{T}$ regarded as a 2-category with identity 2-cells. The \textbf{cocharacter} of $\rho$ is the composite
    \[\Lambda BG\xto{\cob{\rho}} \cob{\B} \xto{\overline{\chi}}\T,\] where $\cob{\rho}(g) = \rho(g)$ and $\cob{\rho} = \overline{\gamma}$.
\end{defn}

\begin{ex}\label{ex:GK}
    Let $\rho: BG \to \B$ be a 2-representation, where $\B$ is a closed bicategory. Let the coshadow in the definition of $\overline{\rchi}$ be $\Tr$, the categorical trace as defined in \cite[3.1]{GK}, with target category $T$. The categorical character, $\Tr(\rho)$, from \cite[4.8]{GK} is given by the commutative diagram: \[\xymatrix{\Lambda BG \ar[rr]^{\Tr(\rho)} \ar[dr]_{\overline{\Lambda}\rho} & & \T\\
    & \cb \ar[ur]_{\overline{\rchi}} &}\] On the level of 1-cells, $\overline{\rchi}(\overline{\gamma}) = \co(\overline{\gamma})$. 
\end{ex}

Just as \cref{lambda_of_a_composite} implies that character commutes with restriction, \cref{colambda_of_a_composite} implies that the cocharacter commutes with restriction.

\section{Induced representations and their characters}\label{sec:induction}

There are four cases to consider here: each of induction and coinduction with traces and cotraces.  
Traces and cotraces share formal properties and induction (resp. coinduction) is a initial object in an under category (resp.  terminal object in an over category) so we will consider only the case of induction with traces and leave the remaining cases to the interested reader. 

Recall from \cref{def:induced_rep} that if ${\rhov}$ is an $\mathcal{C}$-2-representation  and $\iota\colon \mathcal{C}\to \mathcal{D}$ is a strong 2-functor, then the {\bf induced $\mathcal{D}$-2-representation} $I_\mathcal{C}^\mathcal{D}{\rhov}$ is the initial object in the under bicategory ${\rhov}\downarrow \iota^*$. 

Since the proofs of the main results in this section, \cref{intro:induced,thm:trace_induced},  primarily follow from carefully identifying the bicategories in question,  we first unravel the definitions.

\begin{defn}\label{def:comma_bicat}\cite[I.2.5]{gray}
    Let 
% https://q.uiver.app/#q=WzAsMyxbMSwwLCJcXG1hdGhjYWx7Q30iXSxbMCwxLCJcXG1hdGhjYWx7RH0iXSxbMSwxLCJcXG1hdGhjYWx7RX0iXSxbMCwyLCJGIl0sWzEsMiwiRyIsMl1d
\[\begin{tikzcd}
	& {\mathcal{X}} \\
	{\mathcal{Y}} & {\mathcal{Z}}
	\arrow["F", from=1-2, to=2-2]
	\arrow["G"', from=2-1, to=2-2]
\end{tikzcd}\]
be a diagram of strong 2-functors.  Then  the {\bf comma bicategory}, $F\downarrow G$, has 
\begin{enumerate}
    \item 0-cells triples 
    \[(x, y, \zeta\colon F(x)\to G(y))\] 
    where $x$ is a 0-cell of $\mathcal{X}$, $y$ is a 0-cell of $\mathcal{Y}$ and $\zeta$ is a 1-cell in $\mathcal{Z}$.
    \item 1-cells $(x, y, \zeta\colon F(x)\to G(y)) \to (x', y', \zeta'\colon F(x')\to G(y'))$ are triples 
    \[(\phi\colon x\to x', \psi\colon y\to y', \alpha\colon  \zeta\odot G(\psi) \Rightarrow F(\phi)\odot \zeta')\]
    \item 2-cells $(\phi,\psi,\alpha)\to (\phi',\psi',\alpha')$ are pairs 
    \[(\beta\colon \phi\Rightarrow \phi', \gamma\colon \psi\Rightarrow \psi')\] so that 
the following diagram commutes.
% https://q.uiver.app/#q=WzAsNCxbMCwwLCJcXHpldGFcXG9kb3QgRyhcXHBzaSkiXSxbMSwwLCJGKFxccGhpKVxcb2RvdCBcXHpldGEnIl0sWzEsMSwiRihcXHBoaScpXFxvZG90IFxcemV0YSciXSxbMCwxLCJcXHpldGFcXG9kb3QgRyhcXHBzaScpIl0sWzAsMSwiXFxhbHBoYSIsMCx7ImxldmVsIjoyfV0sWzEsMiwiRihcXGJldGEpXFxvZG90IFxcaWQiLDAseyJsZXZlbCI6Mn1dLFswLDMsIlxcaWRcXG9kb3QgRyhcXGdhbW1hKSIsMix7ImxldmVsIjoyfV0sWzMsMiwiXFxhbHBoYSciLDIseyJsZXZlbCI6Mn1dXQ==
\begin{equation}\label{eq:comma_bicat_2_cell}
\begin{tikzcd}
	{\zeta\odot G(\psi)} & {F(\phi)\odot \zeta'} \\
	{\zeta\odot G(\psi')} & {F(\phi')\odot \zeta'}
	\arrow["\alpha", Rightarrow,  from=1-1, to=1-2]
	\arrow["{\id\odot G(\gamma)}"', Rightarrow,  from=1-1, to=2-1]
	\arrow["{F(\beta)\odot \id}", Rightarrow,  from=1-2, to=2-2]
	\arrow["{\alpha'}"', Rightarrow,  from=2-1, to=2-2]
\end{tikzcd}
\end{equation}
\end{enumerate}
\end{defn}

We now consider  \cref{def:comma_bicat} for functor bicategories and further expand the transformations since we will more direct access to them in \cref{lem:0:cell:functor,lem:1:cell:functor}.  

\begin{ex}
\label{ex:under:2}
Let ${\rhov}\colon \mathcal{C}\to \mathcal{B}$ and $\iota\colon \mathcal{C}\to \mathcal{D}$
be strong 2-functors and $\Fun(\mathcal{D},\mathcal{B})$ be the bicategory of strong functors, strong transformations, and modifications.  Consider the diagram
% https://q.uiver.app/#q=WzAsMyxbMSwwLCJcXGFzdCJdLFswLDEsIlxcRnVuKFxcbWF0aGNhbHtEfSxcXG1hdGhjYWx7Qn0pIl0sWzEsMSwiXFxGdW4oXFxtYXRoY2Fse0N9LFxcbWF0aGNhbHtCfSkiXSxbMCwyLCJWIl0sWzEsMiwiXFxpb3RhXioiLDJdXQ==
\[\begin{tikzcd}
	& \ast \\
	{\Fun(\mathcal{D},\mathcal{B})} & {\Fun(\mathcal{C},\mathcal{B})}
	\arrow["{\rhov}", from=1-2, to=2-2]
	\arrow["{\iota^*}"', from=2-1, to=2-2]
\end{tikzcd}\]
Then ${\rhov}\downarrow \iota^*$ is the bicategory where 
\begin{enumerate}
    \item \label{ex:under:2:it:0}
    a 0-cell is a tuple consisting of 
    \begin{enumerate}
            \item \label{ex:under:2:it:0:1} a strong functor ${\sigma}\colon \mathcal{D}\to \mathcal{B}$, 
            \item a 1-cell  $\zeta_c\colon {\rhov}(c)\to {\sigma}(\iota(c))$ in $\mathcal{B}$ for each 0-cell $c$ in $\mathcal{C}$,
            and 
            \item an isomorphism 2-cell 
\[\zeta_f\colon \zeta_c\odot {\sigma}(\iota(f)) \Rightarrow {\rhov}(f)\odot \zeta_{c'}\] in $\mathcal{B}$ for each 1-cell $f\colon c\to c'$ in $\mathcal{C}$
so that 
\begin{enumerate}
\item the diagrams in \cref{ex:under:2:it:0:1:fig:compose,ex:under:2:it:0:1:fig:unit} commute 
and 
\item the diagram in \cref{ex:under:2:nat_1} commutes for each 2-cell $\vartheta\colon f\Rightarrow f'$ in $\mathcal{D}$.
\end{enumerate} 
           
    \end{enumerate}
    \item\label{ex:under:2:it:1} a 1-cell $({\sigma},\zeta)\to ({\sigma}',\zeta')$ is a tuple consisting of 
    \begin{enumerate}
        \item a 1-cell $\psi_d\colon {\sigma}(d)\to {\sigma}'(d)$ in $\mathcal{B}$ for each 0-cell $d\in \mathcal{D}$,
        \item an isomorphism 2-cell 
\[\psi_g\colon \psi_d\odot {\sigma}'(g)\Rightarrow {\sigma}(g)\odot \psi_{d'} \] in $\mathcal{B}$ for each 1-cell $g\colon d\to d'$ in $\mathcal{D}$
        so that 
\begin{enumerate}
\item the diagrams in \cref{ex:under:2:it:1:fig:compose,ex:under:2:it:1:fig:unit} commute and 
\item the diagram in \cref{ex:under:2:nat_2} commutes for all 2-cells  $\vartheta\colon g\Rightarrow g'$ in $\mathcal{D}$. 
\end{enumerate}
        \item a 2-cell 
\[\alpha_c\colon \zeta_c\odot \psi_{\iota(c)}\Rightarrow \zeta'_c \] in $\mathcal{B}$ for each 0-cell $c$ of $\mathcal{C}$
        so that the diagram in \cref{ex:under:2:it:1:fig:modification} commutes. 
    \end{enumerate}
    \item \label{ex:under:2:it:3}a 2-cell $(\psi,\beta)\to (\psi',\beta')$ is a 2-cell 
\[\gamma_d\colon \psi_d\Rightarrow\psi'_d\]
 in $\mathcal{B}$ for each 0-cell $d$ of $\mathcal{D}$
         so that the diagrams in \cref{ex:under:2:it:2:fig:modification,ex:under:2:it:2:fig:comma} commute. 
\end{enumerate}
\end{ex}

\begin{figure}
     \centering
     \begin{subfigure}[b]{0.25\textwidth}
\adjustbox{scale=1,center}{
% https://q.uiver.app/#q=WzAsNCxbMCwwLCJcXHpldGFfY1xcb2RvdCBcXHNpZ21hIFxcaW90YShmKSJdLFswLDEsIlxcemV0YV9jXFxvZG90IFxcc2lnbWEgXFxpb3RhKGYnKSJdLFsxLDAsIlxccmhvKGYpXFxvZG90IFxcemV0YV97Yyd9Il0sWzEsMSwiXFxyaG8oZicpXFxvZG90IFxcemV0YV97Yyd9Il0sWzAsMSwiXFxpZFxcb2RvdCBcXHNpZ21hXFxpb3RhKFxcdmFydGhldGEpIiwyLHsibGV2ZWwiOjJ9XSxbMCwyLCJcXHpldGFfZiIsMCx7ImxldmVsIjoyfV0sWzEsMywiXFx6ZXRhX3tmJ30iLDIseyJsZXZlbCI6Mn1dLFsyLDMsIlxccmhvKFxcdmFydGhldGEpXFxvZG90IFxcaWQiLDAseyJsZXZlbCI6Mn1dXQ==
\begin{tikzcd}
	{\zeta_c\odot \sigma \iota(f)} & {\rho(f)\odot \zeta_{c'}} \\
	{\zeta_c\odot \sigma \iota(f')} & {\rho(f')\odot \zeta_{c'}}
	\arrow["{\zeta_f}", Rightarrow,  from=1-1, to=1-2]
	\arrow["{\id\odot \sigma\iota(\vartheta)}", Rightarrow,  from=1-1, to=2-1]
	\arrow["{\rho(\vartheta)\odot \id}", Rightarrow,  from=1-2, to=2-2]
	\arrow["{\zeta_{f'}}"', Rightarrow,  from=2-1, to=2-2]
\end{tikzcd}
}
\caption{}
         \label{ex:under:2:nat_1}
     \end{subfigure}
     \hspace{1in}
     \begin{subfigure}[b]{0.25\textwidth}
\adjustbox{scale=1,center}{
% https://q.uiver.app/#q=WzAsNCxbMCwwLCJcXHBzaV9kXFxvZG90IHtcXHJob3d9JyhnKSAiXSxbMSwwLCJ7XFxyaG93fShnKVxcb2RvdCBcXHBzaV97ZCd9ICJdLFsxLDEsIntcXHJob3d9KGcnKVxcb2RvdCBcXHBzaV97ZCd9ICJdLFswLDEsIiBcXHBzaV9kXFxvZG90IHtcXHJob3d9JyhnJykiXSxbMCwxLCJcXHBzaV9nIiwwLHsibGV2ZWwiOjJ9XSxbMywyLCJcXHBzaV97Zyd9IiwyLHsibGV2ZWwiOjJ9XSxbMCwzLCJcXGlkXFxvZG90IFxccmhvdycoXFx2YXJ0aGV0YSkiLDIseyJsZXZlbCI6Mn1dLFsxLDIsIlxccmhvdyhcXHZhcnRoZXRhKVxcb2RvdCBcXGlkIiwwLHsibGV2ZWwiOjJ9XV0=
\begin{tikzcd}
	{\psi_d\odot {\sigma}'(g) } & {{\sigma}(g)\odot \psi_{d'} } \\
	{ \psi_d\odot {\sigma}'(g')} & {{\sigma}(g')\odot \psi_{d'} }
	\arrow["{\psi_g}", Rightarrow,  from=1-1, to=1-2]
	\arrow["{\id\odot \sigma'(\vartheta)}"', Rightarrow,  from=1-1, to=2-1]
	\arrow["{\sigma(\vartheta)\odot \id}", Rightarrow,  from=1-2, to=2-2]
	\arrow["{\psi_{g'}}"', Rightarrow,  from=2-1, to=2-2]
\end{tikzcd}}
\caption{}
         \label{ex:under:2:nat_2}
     \end{subfigure}
\caption{Naturality conditions}\label{ex:nat_diagrams}
\end{figure}

\begin{figure}
     \centering
     \begin{subfigure}[b]{0.53\textwidth}
\adjustbox{scale=.9,center}{% https://q.uiver.app/#q=WzAsNixbMCwwLCJcXHpldGFfY1xcb2RvdCBcXHNpZ21hIChcXGlvdGEoZikpXFxvZG90IFxcc2lnbWEgKFxcaW90YShmJykpIl0sWzAsMSwiXFx6ZXRhX2NcXG9kb3QgXFxzaWdtYSAoXFxpb3RhKGYpXFxvZG90IFxcaW90YShmJykpIl0sWzAsMiwiXFx6ZXRhX2NcXG9kb3QgXFxzaWdtYSAoXFxpb3RhKGZcXG9kb3QgZicpKSJdLFsxLDAsIlxccmhvKGYpXFxvZG90IFxcemV0YV97Yyd9XFxvZG90IFxcc2lnbWEgKFxcaW90YShmJykpIl0sWzIsMCwiXFxyaG8oZilcXG9kb3QgXFxyaG8oZicpXFxvZG90IFxcemV0YV97YycnfSJdLFsyLDIsIlxccmhvKGZcXG9kb3QgZicpXFxvZG90IFxcemV0YV97YycnfSJdLFswLDEsIiIsMCx7ImxldmVsIjoyfV0sWzAsMywiXFx6ZXRhX2ZcXG9kb3QgMSIsMCx7ImxldmVsIjoyfV0sWzEsMiwiIiwwLHsibGV2ZWwiOjJ9XSxbMiw1LCJcXHpldGFfe2ZcXG9kb3QgZ30iLDAseyJsZXZlbCI6Mn1dLFszLDQsIjFcXG9kb3QgXFx6ZXRhX2ciLDAseyJsZXZlbCI6Mn1dLFs0LDUsIiIsMCx7ImxldmVsIjoyfV1d
\begin{tikzcd}
	{\zeta_c\odot \sigma (\iota(f))\odot \sigma (\iota(f'))} & {\rho(f)\odot \zeta_{c'}\odot \sigma (\iota(f'))} & {\rho(f)\odot \rho(f')\odot \zeta_{c''}} \\
	{\zeta_c\odot \sigma (\iota(f)\odot \iota(f'))} \\
	{\zeta_c\odot \sigma (\iota(f\odot f'))} && {\rho(f\odot f')\odot \zeta_{c''}}
	\arrow["{\zeta_f\odot 1}", Rightarrow,  from=1-1, to=1-2]
	\arrow[Rightarrow,  from=1-1, to=2-1]
	\arrow["{1\odot \zeta_g}", Rightarrow,  from=1-2, to=1-3]
	\arrow[Rightarrow,  from=1-3, to=3-3]
	\arrow[Rightarrow,  from=2-1, to=3-1]
	\arrow["{\zeta_{f\odot g}}", Rightarrow,  from=3-1, to=3-3]
\end{tikzcd}
}
         \caption{}
         \label{ex:under:2:it:0:1:fig:compose}
     \end{subfigure}
     \hfill 
     \begin{subfigure}[b]{0.38\textwidth}
         \centering
\adjustbox{scale=.9,center}{
% https://q.uiver.app/#q=WzAsNixbMCwwLCJcXHpldGFfY1xcb2RvdCBVX3tcXHNpZ21hKFxcaW90YShjKSl9Il0sWzEsMCwiXFx6ZXRhX2MiXSxbMCwxLCJcXHpldGFfY1xcb2RvdCBcXHNpZ21hKFVfe1xcaW90YShjKX0pIl0sWzAsMiwiXFx6ZXRhX2NcXG9kb3QgXFxzaWdtYShcXGlvdGEoVV97Y30pKSJdLFsyLDIsIlxccmhvKFVfYylcXG9kb3QgXFx6ZXRhX3tjfSJdLFsyLDAsIlVfe1xccmhvKGMpfVxcb2RvdCBcXHpldGFfYyJdLFswLDEsInIiLDAseyJsZXZlbCI6Mn1dLFswLDIsIiIsMCx7ImxldmVsIjoyfV0sWzIsMywiIiwwLHsibGV2ZWwiOjJ9XSxbMyw0LCJcXHpldGFfe1VfY30iLDAseyJsZXZlbCI6Mn1dLFs1LDEsIlxcZWxsIiwyLHsibGV2ZWwiOjJ9XSxbNSw0LCIiLDAseyJsZXZlbCI6Mn1dXQ==
\begin{tikzcd}
	{\zeta_c\odot U_{\sigma(\iota(c))}} & {\zeta_c} & {U_{\rho(c)}\odot \zeta_c} \\
	{\zeta_c\odot \sigma(U_{\iota(c)})} \\
	{\zeta_c\odot \sigma(\iota(U_{c}))} && {\rho(U_c)\odot \zeta_{c}}
	\arrow["r", Rightarrow,  from=1-1, to=1-2]
	\arrow[Rightarrow,  from=1-1, to=2-1]
	\arrow["\ell"', Rightarrow,  from=1-3, to=1-2]
	\arrow[Rightarrow,  from=1-3, to=3-3]
	\arrow[Rightarrow,  from=2-1, to=3-1]
	\arrow["{\zeta_{U_c}}", Rightarrow,  from=3-1, to=3-3]
\end{tikzcd}}
         \caption{}
         \label{ex:under:2:it:0:1:fig:unit}
     \end{subfigure}

     \begin{subfigure}[b]{0.53\textwidth}
\adjustbox{scale=.9,center}{
% https://q.uiver.app/#q=WzAsNSxbMCwwLCJcXHBzaV9kXFxvZG90IHtcXHNpZ21hfScoZylcXG9kb3Qge1xcc2lnbWF9JyhnJykiXSxbMCwxLCJcXHBzaV9kXFxvZG90IHtcXHNpZ21hfScoZ1xcb2RvdCBnJykiXSxbMSwwLCJ7XFxzaWdtYX0oZylcXG9kb3QgXFxwc2lfe2QnfVxcb2RvdCB7XFxzaWdtYX0nKGcnKSJdLFsyLDAsIntcXHNpZ21hfShnKVxcb2RvdCB7XFxzaWdtYX0oZycpXFxvZG90IFxccHNpX3tkJyd9Il0sWzIsMSwie1xcc2lnbWF9KGdcXG9kb3QgZycpXFxvZG90IFxccHNpX3tkJyd9Il0sWzAsMSwiIiwwLHsibGV2ZWwiOjJ9XSxbMCwyLCJcXHBzaV9nXFxvZG90IDEiLDAseyJsZXZlbCI6Mn1dLFsxLDQsIlxccHNpX3tnXFxvZG90IGcnfSIsMix7ImxldmVsIjoyfV0sWzIsMywiMVxcb2RvdCBcXHBzaV97Zyd9IiwwLHsibGV2ZWwiOjJ9XSxbMyw0LCIiLDAseyJsZXZlbCI6Mn1dXQ==
\begin{tikzcd}
	{\psi_d\odot {\sigma}'(g)\odot {\sigma}'(g')} & {{\sigma}(g)\odot \psi_{d'}\odot {\sigma}'(g')} & {{\sigma}(g)\odot {\sigma}(g')\odot \psi_{d''}} \\
	{\psi_d\odot {\sigma}'(g\odot g')} && {{\sigma}(g\odot g')\odot \psi_{d''}}
	\arrow["{\psi_g\odot 1}", Rightarrow,  from=1-1, to=1-2]
	\arrow[Rightarrow,  from=1-1, to=2-1]
	\arrow["{1\odot \psi_{g'}}", Rightarrow,  from=1-2, to=1-3]
	\arrow[Rightarrow,  from=1-3, to=2-3]
	\arrow["{\psi_{g\odot g'}}"', Rightarrow,  from=2-1, to=2-3]
\end{tikzcd}}
    \caption{}\label{ex:under:2:it:1:fig:compose}
     \end{subfigure}
         \hfill 
     \begin{subfigure}[b]{0.38\textwidth}
\adjustbox{scale=.9,center}{
% https://q.uiver.app/#q=WzAsNSxbMCwwLCJcXHBzaV9kXFxvZG90IFVfe3tcXHNpZ21hfScoZCl9Il0sWzEsMCwiXFxwc2lfZCJdLFsyLDAsIlVfe3tcXHNpZ21hfShkKX1cXG9kb3QgXFxwc2lfZCJdLFswLDEsIlxccHNpX2RcXG9kb3Qge1xcc2lnbWF9JyhVX3tkfSkiXSxbMiwxLCJ7XFxzaWdtYX0oVV9kKVxcb2RvdCBcXHBzaV9kIl0sWzAsMSwiciIsMCx7ImxldmVsIjoyfV0sWzAsMywiIiwwLHsibGV2ZWwiOjJ9XSxbMiwxLCJcXGVsbCIsMix7ImxldmVsIjoyfV0sWzIsNCwiIiwwLHsibGV2ZWwiOjJ9XSxbMyw0LCJcXHBzaV97VV9kfSIsMix7ImxldmVsIjoyfV1d
\begin{tikzcd}
	{\psi_d\odot U_{{\sigma}'(d)}} & {\psi_d} & {U_{{\sigma}(d)}\odot \psi_d} \\
	{\psi_d\odot {\sigma}'(U_{d})} && {{\sigma}(U_d)\odot \psi_d}
	\arrow["r", Rightarrow,  from=1-1, to=1-2]
	\arrow[Rightarrow,  from=1-1, to=2-1]
	\arrow["\ell"', Rightarrow,  from=1-3, to=1-2]
	\arrow[Rightarrow,  from=1-3, to=2-3]
	\arrow["{\psi_{U_d}}"', Rightarrow,  from=2-1, to=2-3]
\end{tikzcd}
}
     \caption{}\label{ex:under:2:it:1:fig:unit}
     \end{subfigure}
    \hfill 
\caption{Transformation axioms}\label{ex:trans_diagrams}
\end{figure}

\begin{figure}
     \centering
\hfill 
          \begin{subfigure}[b]{0.35\textwidth}
\adjustbox{scale=.9,center}{
% https://q.uiver.app/#q=WzAsNCxbMCwwLCJcXHBzaV9kXFxvZG90IHtcXHNpZ21hfScoZykiXSxbMCwxLCJcXHBzaV9kJ1xcb2RvdCB7XFxzaWdtYX0nKGcpIl0sWzEsMCwie1xcc2lnbWF9KGcpXFxvZG90IFxccHNpX3tkJ30iXSxbMSwxLCJ7XFxzaWdtYX0oZylcXG9kb3QgXFxwc2knX3tkJ30iXSxbMCwxLCJcXGdhbW1hX2RcXG9kb3QgMSIsMix7ImxldmVsIjoyfV0sWzAsMiwiXFxwc2lfZyIsMCx7ImxldmVsIjoyfV0sWzEsMywiXFxwc2knX2ciLDIseyJsZXZlbCI6Mn1dLFsyLDMsIjFcXG9kb3QgXFxnYW1tYV97ZCd9IiwwLHsibGV2ZWwiOjJ9XV0=
\begin{tikzcd}
	{\psi_d\odot {\sigma}'(g)} & {{\sigma}(g)\odot \psi_{d'}} \\
	{\psi_d'\odot {\sigma}'(g)} & {{\sigma}(g)\odot \psi'_{d'}}
	\arrow["{\psi_g}", Rightarrow,  from=1-1, to=1-2]
	\arrow["{\gamma_d\odot 1}"', Rightarrow,  from=1-1, to=2-1]
	\arrow["{1\odot \gamma_{d'}}", Rightarrow,  from=1-2, to=2-2]
	\arrow["{\psi'_g}"', Rightarrow,  from=2-1, to=2-2]
\end{tikzcd}
}
\caption{}\label{ex:under:2:it:2:fig:modification}
     \end{subfigure}
\hfill 
     \begin{subfigure}[b]{0.55\textwidth}
\adjustbox{scale=.9,center}{
% https://q.uiver.app/#q=WzAsNSxbMCwwLCJcXHpldGFfY1xcb2RvdCBcXHBzaV97XFxpb3RhKGMpfVxcb2RvdCBcXHNpZ21hJyhcXGlvdGEoZikpIl0sWzAsMSwiXFx6ZXRhX2MnXFxvZG90IFxcc2lnbWEnKFxcaW90YShmKSkiXSxbMSwwLCJcXHpldGFfY1xcb2RvdCBcXHNpZ21hKFxcaW90YShmKSlcXG9kb3QgXFxwc2lfe1xcaW90YShjJyl9Il0sWzIsMCwiXFxyaG8oZilcXG9kb3QgXFx6ZXRhX3tjJ31cXG9kb3QgXFxwc2lfe1xcaW90YShjJyl9Il0sWzIsMSwiXFxyaG8oZilcXG9kb3QgXFx6ZXRhJ197Yyd9Il0sWzAsMSwiXFxhbHBoYV9jXFxvZG90IDEiLDAseyJsZXZlbCI6Mn1dLFswLDIsIjFcXG9kb3QgXFxwc2lfe1xcaW90YShmKX0iLDAseyJsZXZlbCI6Mn1dLFsxLDQsIlxcemV0YSdfZiIsMix7ImxldmVsIjoyfV0sWzIsMywiXFx6ZXRhX2ZcXG9kb3QgMSIsMCx7ImxldmVsIjoyfV0sWzMsNCwiMVxcb2RvdCBcXGFscGhhX3tjJ30iLDAseyJsZXZlbCI6Mn1dXQ==
\begin{tikzcd}
	{\zeta_c\odot \psi_{\iota(c)}\odot \sigma'(\iota(f))} & {\zeta_c\odot \sigma(\iota(f))\odot \psi_{\iota(c')}} & {\rho(f)\odot \zeta_{c'}\odot \psi_{\iota(c')}} \\
	{\zeta_c'\odot \sigma'(\iota(f))} && {\rho(f)\odot \zeta'_{c'}}
	\arrow["{1\odot \psi_{\iota(f)}}", Rightarrow,  from=1-1, to=1-2]
	\arrow["{\alpha_c\odot 1}", Rightarrow,  from=1-1, to=2-1]
	\arrow["{\zeta_f\odot 1}", Rightarrow,  from=1-2, to=1-3]
	\arrow["{1\odot \alpha_{c'}}", Rightarrow,  from=1-3, to=2-3]
	\arrow["{\zeta'_f}"', Rightarrow,  from=2-1, to=2-3]
\end{tikzcd}
}
\caption{}\label{ex:under:2:it:1:fig:modification}
     \end{subfigure}
\hfill
\caption{Modification axioms}
\end{figure}

\begin{figure}
     \centering
     % https://q.uiver.app/#q=WzAsMyxbMCwwLCJcXGVwc2lsb25fY1xcb2RvdCBcXHBzaV97XFxpb3RhKGMpfSJdLFsxLDEsIlxcZXBzaWxvbidfYyJdLFswLDIsIlxcZXBzaWxvbl9jXFxvZG90IFxccHNpJ197XFxpb3RhKGMpfSJdLFswLDEsIlxcYmV0YV9jIiwwLHsibGV2ZWwiOjJ9XSxbMCwyLCIxXFxvZG90IFxcZ2FtbWFfe1xcaW90YShjKX0iLDIseyJsZXZlbCI6Mn1dLFsyLDEsIlxcYmV0YSdfYyIsMix7ImxldmVsIjoyfV1d
\begin{tikzcd}[row sep=0in]
	{\epsilon_c\odot \psi_{\iota(c)}} \\
	& {\epsilon'_c} \\
	{\epsilon_c\odot \psi'_{\iota(c)}}
	\arrow["{\beta_c}", Rightarrow, from=1-1, to=2-2]
	\arrow["{1\odot \gamma_{\iota(c)}}"', Rightarrow, from=1-1, to=3-1]
	\arrow["{\beta'_c}"', Rightarrow, from=3-1, to=2-2]
\end{tikzcd}
	\caption{The condition in \eqref{eq:comma_bicat_2_cell}}
        \label{fig:ex:under:2}\label{ex:under:2:it:2:fig:comma}
\end{figure}

\begin{rmk}
Let $F,G\colon \mathcal{B}\to  \mathcal{C}$ be strong functors of bicategories and $\xi\colon F\to G$ be a transformation.  
If a 1-cell $f\colon b\to b'$ in $\mathcal{B}$ is dualizable with dual $f^*\colon b'\to b$, then the mate \cite[p. 5]{Ponto_2012} of $\xi_{f^*}$ is a 2-cell 
\[F(f)\odot \xi_{b'}\Rightarrow \xi_b\odot G(f)\]
and this is the inverse of $\xi_f$.  
This is essentially \cite[9.4]{Ponto_2012}.

In particular, if all 1-cells in $\mathcal{C}$ and $\mathcal{D}$ are dualizable in \cref{ex:under:2}, the isomorphism conditions on $\zeta$ and $\psi$ in \cref{ex:under:2} are automatically satisfied.

\end{rmk}

The following definition looks forward to \cref{intro:induced,thm:trace_induced} and is a generalization of the dualizability condition in \cref{thm:intro:induced_classical}.
\begin{defn}\label{def:conditions_for_induction}
If all 1-cells of $\mathcal{C}$ and $\mathcal{D}$ are dualizable, $({\rhov}\downarrow \iota^*)^d$ is the subbicategory of ${\rhov}\downarrow \iota^*$ consisting of 
\begin{itemize}
\item 0-cells where the 1-cell  $\zeta_c\colon {\rhov}(c)\to {\sigma}(\iota(c))$ in $\mathcal{B}$ is dualizable 
 for each 0-cell $c$ in $\mathcal{C}$
and
\item 1-cells where  the 1-cell $\psi_d\colon {\sigma}(d)\to {\sigma}'(d)$ in $\mathcal{B}$ is dualizable for each 0-cell $d\in \mathcal{D}$,
\end{itemize}
\end{defn}
In the case where $\mathcal{C}=BH$ and $\mathcal{D}=BG$, the dualizability conditions in \cref{ex:under:2} on the 1-cells of $\mathcal{C}$ and $\mathcal{D}$ are satisfied.  So this definition requires the 1-cells \[\zeta_\ast\colon {\rhov}(\ast)\to {\sigma}(\ast)\]
and 
\[\psi_\ast\colon {\sigma}(\ast)\to {\sigma}'(\ast)\]
for the single 0-cells $\ast$ of $\mathcal{C}$ and $\mathcal{D}$ are dualizable.  
This condition should be regarded as analogous to the familiar finiteness conditions in the compatibility of induction with classical characters. 

\cref{ex:under:2} gives a very explicit description of the over category used to define the induced 2-representations.  We now turn to the induced characters. 

\begin{ex}\label{ex:under:3}
Let ${\rhov}$ and  $\iota$ be as in \cref{ex:under:2} and $\chi$ be as in \cref{thm:character}.
% https://q.uiver.app/#q=WzAsNCxbMSwwLCJcXGFzdCJdLFsxLDEsIlxcRnVuKFxcTGFtYmRhIFxcbWF0aGNhbHtDfSxcXExhbWJkYSBcXG1hdGhjYWx7Qn0pIl0sWzEsMiwiXFxGdW4oXFxMYW1iZGEgXFxtYXRoY2Fse0N9LFxcbWF0aGNhbHtUfSkiXSxbMCwyLCJcXEZ1bihcXExhbWJkYSBcXG1hdGhjYWx7RH0sXFxtYXRoY2Fse1R9KSJdLFswLDEsIlxcTGFtYmRhIFYiLDJdLFsxLDIsIlxcY2hpXyoiLDJdLFszLDIsIihcXExhbWJkYSBcXGlvdGEpXioiXV0=
\[\begin{tikzcd}
	& \ast \\
	& {\Fun(\Lambda \mathcal{C},\Lambda \mathcal{B})} \\
	{\Fun(\Lambda \mathcal{D},\mathcal{T})} & {\Fun(\Lambda \mathcal{C},\mathcal{T})}
	\arrow["{\Lambda {\rhov}}"', from=1-2, to=2-2]
	\arrow["{\chi_*}"', from=2-2, to=3-2]
	\arrow["{(\Lambda \iota)^*}", from=3-1, to=3-2]
\end{tikzcd}\]
Then $((\chi_*\circ (\Lambda {\rhov}))\downarrow (\Lambda \iota^*))$ is the bicategory where 
\begin{enumerate}
    \item\label{ex:under:3:it:0} a 0-cell is a tuple consisting of 
    \begin{enumerate}
        \item a functor $M\colon \Lambda \mathcal{D}\to \mathcal{T}$,
        \item  a 1-cell 
\[\mu_f\colon \sh{{\rhov}(f)}\to M(\iota(f))\] for every endomorphism 1-cell $f$ of $\mathcal{C}$ (regarded as a 0-cell in $\Lambda \mathcal{C}$), and 
        \item an equality 
\[\mu_f\odot  M(\iota(k,\beta))=\tr(\phi^{-1}\odot {\rhov}(\beta)\odot \phi)\odot \mu_{f'}\]
        for each 2-cell $\beta\colon f\odot k\Rightarrow  k\odot f'$ in $\mathcal{C}$ where $k$ is dualizable.  (The pair $(k,\beta)$ is a 1-cell of $\Lambda \mathcal{C}$.)
        
        Recall that $\phi\colon {\rhov}(f)\odot {\rhov}(k)\Rightarrow {\rhov}(f\odot k)$ is a structure map for ${\rhov}$.
    \end{enumerate}
    \item\label{ex:under:3:it:1}  a 1-cell $(M,\mu)\to (M,\mu')$ is a tuple consisting of 
    \begin{enumerate}
        \item a 1-cell 
\[\nu_g\colon M(g)\to M'(g)\] in $\mathcal{T}$ for each endomorphism 1-cell $g$ of $\mathcal{D}$ (regarded as a 0-cell of $\Lambda \mathcal{D}$),
                \item an equality \[\nu_g\odot M'(\rho) = M(\rho)\odot \nu_{g'}\] for each 2-cell $\rho\colon g\odot h\to h\odot g'$ in $\mathcal{D}$ where $h$ is dualizable, 
        \item an equality $\mu_f\circ \nu_{\iota(f)}=\mu'_f $ for each endomorphism 1-cell $f$ of $\mathcal{C}$.
    \end{enumerate}
    \item a 2-cell $\nu\Rightarrow \nu'$ is equalities $\nu_g=\nu_{g}'$ for all endomorphism 1-cells in  $g$ in $\mathcal{D}$.
\end{enumerate}
\end{ex}
If there is a 2-cell  $\alpha\colon k\to k'$ from $(k,\beta)$ to $(k',\beta')$ then 
$M(k,\beta)=M(k',\beta')$ since $\mathcal{T}$ has only identity 2-cells.  In particular, 
the analogous diagrams to those in \cref{ex:nat_diagrams,ex:trans_diagrams} automatically commute since they are composed of identity maps. We see the same result for the conditions coming from modifications.

\begin{rmk}
    The analog of \cref{ex:under:3} for the coshadow and cotrace is very similar.  The only differences are  
    \begin{enumerate}
        \item $\sh{{\rhov}(f)}$ is replaced by $\coshh{{\rhov}(f)}$,
        \item $\tr({\rhov}(\beta))$ is replaced by $\co({\rhov}(\beta))$, 
        \item $\beta\colon f\odot k\to k\odot f'$ is replaced by $\beta \colon k\triangleright f\to f'\triangleleft k$, 
        and 
        \item all instances of $\Lambda$ are replaced by $\bar{\Lambda} $.
    \end{enumerate}
\end{rmk}

We can now start to define a functor 
    \[X\colon ({\rhov}\downarrow \iota^*)^d\to (\Lambda {\rhov}\downarrow \Lambda \iota^*)\]
    that extends the character defined in \cref{lab:def:character}.

\begin{lem}\label{lem:0:cell:functor} Let $({\rho},\zeta)$ be a 0-cell  as in \cref{ex:under:2}(\ref{ex:under:2:it:0}).  
If \begin{enumerate}
    \item all 1-cells in $\mathcal{C}$ are dualizable and 
    \item  $\zeta_c$ is dualizable for all 0-cells $c$ in $\mathcal{C}$
\end{enumerate} then  
    \begin{enumerate}[resume]
        \item the functor $\Lambda \mathcal{D}\xto{\Lambda {\sigma}}\Lambda \mathcal{B}\xto{\chi}\mathcal{T}$ and 
        \item the 1-cells 
\[\tr(\zeta_f^{-1})\colon\sh{{\rhov}(f)} \to \sh{{\sigma}(\iota(f))}\] for $\zeta_f\colon \zeta_c\odot {\sigma}(\iota(f)) \Rightarrow {\rhov}(f)\odot \zeta_{c}$ associated to each endomorphism 1-cell $f$
    \end{enumerate}
    define a 0-cell as in \cref{ex:under:3}(\ref{ex:under:3:it:0}).
\end{lem}

Note the conditions here are exactly the conditions in \cref{def:conditions_for_induction}.

\begin{proof}  
The only verification needed  is to confirm that the following diagram commutes
% https://q.uiver.app/#q=WzAsNCxbMCwwLCJcXHNoe3tcXHJob30oZil9Il0sWzIsMCwiXFxzaHt7XFxyaG99KGYnKX0iXSxbMCwxLCJcXHNoe3tcXHNpZ21hIFxcaW90YX0oZil9Il0sWzIsMSwiXFxzaHt7XFxzaWdtYSBcXGlvdGF9KGYnKX0iXSxbMCwxLCJcXHRyKFxccGhpXnstMX1cXG9kb3Qge1xccmhvfShcXGJldGEpXFxvZG90IFxccGhpKSJdLFswLDIsIlxcdHIoXFx6ZXRhX2Zeey0xfSkiLDJdLFsxLDMsIlxcdHIoXFx6ZXRhX3tmJ31eey0xfSkiXSxbMiwzLCJcXHRyKFxccGhpXnstMX1cXG9kb3Qge1xcc2lnbWFcXGlvdGF9KFxcYmV0YSlcXG9kb3QgXFxwaGkpIiwyXV0=
\[\begin{tikzcd}
	{\sh{{\rho}(f)}} && {\sh{{\rho}(f')}} \\
	{\sh{{\sigma \iota}(f)}} && {\sh{{\sigma \iota}(f')}}
	\arrow["{\tr(\phi^{-1}\odot {\rho}(\beta)\odot \phi)}", from=1-1, to=1-3]
	\arrow["{\tr(\zeta_f^{-1})}"', from=1-1, to=2-1]
	\arrow["{\tr(\zeta_{f'}^{-1})}", from=1-3, to=2-3]
	\arrow["{\tr(\phi^{-1}\odot {\sigma\iota}(\beta)\odot \phi)}"', from=2-1, to=2-3]
\end{tikzcd}\]
for 2-cells 
$f\odot g\xRightarrow{\beta}g\odot f'$. 

 To see this first note the  diagram in \cref{fig:0:cell:functor} commutes by \cref{ex:under:2:it:0:1:fig:compose}. 
\begin{figure}
% https://q.uiver.app/#q=WzAsMTAsWzAsMCwie1xccmhvfShmKVxcb2RvdCB7XFxyaG99KGspXFxvZG90IFxcemV0YV97YycnfSJdLFs0LDAsIntcXHJob30oZilcXG9kb3QgXFx6ZXRhX3tjJ31cXG9kb3Qge1xcc2lnbWF9XFxpb3RhKGYnKSJdLFs0LDEsIlxcemV0YV97Y31cXG9kb3Qge1xcc2lnbWF9XFxpb3RhKGYpXFxvZG90IHtcXHNpZ21hfVxcaW90YShmJykiXSxbMSwxLCJ7XFxyaG99KGZcXG9kb3QgaylcXG9kb3QgXFx6ZXRhX3tjJyd9Il0sWzMsMSwiXFx6ZXRhX3tjfVxcb2RvdCB7XFxzaWdtYX1cXGlvdGEoZlxcb2RvdCBrKSJdLFsxLDIsIntcXHJob30oa1xcb2RvdCBmJylcXG9kb3QgXFx6ZXRhX3tjJyd9Il0sWzMsMiwiXFx6ZXRhX3tjfVxcb2RvdCB7XFxzaWdtYX1cXGlvdGEoa1xcb2RvdCBmJykiXSxbMCwyLCJ7XFxyaG99KGspXFxvZG90IHtcXHJob30oZicpXFxvZG90IFxcemV0YV97YycnfSJdLFswLDMsIntcXHJob30oaylcXG9kb3QgXFx6ZXRhX3tjJ31cXG9kb3Qge1xcc2lnbWF9XFxpb3RhKGYnKSJdLFs0LDMsIlxcemV0YV97Y31cXG9kb3Qge1xcc2lnbWF9XFxpb3RhKGspXFxvZG90IHtcXHNpZ21hfVxcaW90YShmJykiXSxbMCwxLCJcXGlkXFxvZG90IFxcemV0YV97a30iLDAseyJsZXZlbCI6Mn1dLFswLDMsIiIsMCx7ImxldmVsIjoyfV0sWzAsNywiIiwwLHsibGV2ZWwiOjIsInN0eWxlIjp7ImJvZHkiOnsibmFtZSI6ImRhc2hlZCJ9fX1dLFsxLDIsIlxcemV0YV9mXnstMX1cXG9kb3QgXFxpZCIsMCx7ImxldmVsIjoyfV0sWzIsNCwiIiwwLHsibGV2ZWwiOjJ9XSxbMiw5LCIiLDAseyJsZXZlbCI6Miwic3R5bGUiOnsiYm9keSI6eyJuYW1lIjoiZGFzaGVkIn19fV0sWzMsNCwiXFx6ZXRhX3tmXFxvZG90IGt9IiwwLHsibGV2ZWwiOjJ9XSxbMyw1LCJ7XFxyaG99KFxcYmV0YSlcXG9kb3QgXFxpZCIsMix7ImxldmVsIjoyfV0sWzQsNiwiXFxpZFxcb2RvdCB7XFxzaWdtYX0oXFxiZXRhKSIsMCx7ImxldmVsIjoyfV0sWzUsNiwiXFx6ZXRhX3trXFxvZG90IGYnfSIsMCx7ImxldmVsIjoyfV0sWzcsNSwiIiwwLHsibGV2ZWwiOjJ9XSxbNyw4LCJcXGlkXFxvZG90IFxcemV0YV97Zid9XnstMX0iLDIseyJsZXZlbCI6Mn1dLFs4LDksIlxcemV0YV9rXFxvZG90IFxcaWQiLDIseyJsZXZlbCI6Mn1dLFs5LDYsIiIsMCx7ImxldmVsIjoyfV1d
\[\begin{tikzcd}[column sep =.25in]
	{{\rho}(f)\odot {\rho}(k)\odot \zeta_{c''}} &&&& {{\rho}(f)\odot \zeta_{c'}\odot {\sigma}\iota(f')} \\
	& {{\rho}(f\odot k)\odot \zeta_{c''}} && {\zeta_{c}\odot {\sigma}\iota(f\odot k)} & {\zeta_{c}\odot {\sigma}\iota(f)\odot {\sigma}\iota(f')} \\
	{{\rho}(k)\odot {\rho}(f')\odot \zeta_{c''}} & {{\rho}(k\odot f')\odot \zeta_{c''}} && {\zeta_{c}\odot {\sigma}\iota(k\odot f')} \\
	{{\rho}(k)\odot \zeta_{c'}\odot {\sigma}\iota(f')} &&&& {\zeta_{c}\odot {\sigma}\iota(k)\odot {\sigma}\iota(f')}
	\arrow["{\id\odot \zeta_{k}}", Rightarrow,  from=1-1, to=1-5]
	\arrow[Rightarrow,  from=1-1, to=2-2]
	\arrow[Rightarrow,  dashed, from=1-1, to=3-1]
	\arrow["{\zeta_f^{-1}\odot \id}", Rightarrow,  from=1-5, to=2-5]
	\arrow["{\zeta_{f\odot k}}", Rightarrow,  from=2-2, to=2-4]
	\arrow["{{\rho}(\beta)\odot \id}"', Rightarrow,  from=2-2, to=3-2]
	\arrow["{\id\odot {\sigma}(\beta)}", Rightarrow,  from=2-4, to=3-4]
	\arrow[Rightarrow,  from=2-5, to=2-4]
	\arrow[Rightarrow,  dashed, from=2-5, to=4-5]
	\arrow[Rightarrow,  from=3-1, to=3-2]
	\arrow["{\id\odot \zeta_{f'}^{-1}}"', Rightarrow,  from=3-1, to=4-1]
	\arrow["{\zeta_{k\odot f'}}", Rightarrow,  from=3-2, to=3-4]
	\arrow["{\zeta_k\odot \id}"', Rightarrow,  from=4-1, to=4-5]
	\arrow[Rightarrow, from=4-5, to=3-4]
\end{tikzcd}\]
\caption{The equality in the proof of  \cref{lem:0:cell:functor}}\label{fig:0:cell:functor}
\end{figure}
 By \cite[7.5]{Ponto_2012}, the trace of the left composite 
  in \cref{fig:0:cell:functor} is 
\[\tr(\phi^{-1}\odot {\rho}(\beta)\odot \phi)\circ  \tr(\zeta_{f'}^{-1})\]
 and the trace of the  composite labeled $B$
  is 
\[\tr(\zeta_{f}^{-1})\circ  \tr(\phi^{-1}\odot ({\sigma}\iota\beta)\odot\phi).\] 
Since the horizontal maps are isomorphisms, these traces are equal. 
\end{proof}

\begin{lem}\label{lem:1:cell:functor}
    Let $\psi$ be a 1-cell $({\sigma},\zeta)\to ({\sigma}',\zeta')$ as in \cref{ex:under:2}(\ref{ex:under:2:it:1}).  If 
    \begin{enumerate}
        \item all 1-cells of $\mathcal{D}$ are dualizable and 
        \item $\psi_d$ is dualizable for all 0-cells $d$ in $\mathcal{D}$ 
    \end{enumerate}  then 
    \begin{enumerate}[resume]
        \item the 1-cells 
	\[\tr(\psi_g^{-1})\colon \sh{{\sigma}(g)}\to \sh{{\sigma}'(g)}\] for each endomorphism 1-cell in $\mathcal{D}$
    \end{enumerate}
    define a 1-cell $(\chi\circ \Lambda {\sigma}, \tr(\zeta_-))\to (\chi\circ \Lambda {\sigma}', \tr(\zeta'_-))$ as in \cref{ex:under:3}(\ref{ex:under:3:it:1}).
\end{lem}
The conditions are those from \cref{def:conditions_for_induction}.

\begin{proof}       It is enough to check the equalities in \cref{ex:under:3}\eqref{ex:under:3:it:1}.   Let $\xi\colon g\odot h\Rightarrow h\odot g' $ be a 2-cell in $\mathcal{D}$ where $h$ is dualizable.  
Then verifying the commutativity of 
% https://q.uiver.app/#q=WzAsNCxbMCwwLCJcXHNoe1coZyl9Il0sWzAsMSwiXFxzaHtXJyhnKX0iXSxbMiwwLCJcXHNoe1coZycpfSJdLFsyLDEsIlxcc2h7VycoZycpfSJdLFswLDIsIlxcdHIoXFxwaGleey0xfVxcb2RvdCBXKFxceGkpXFxvZG90IFxccGhpKSJdLFsyLDMsIlxcdHIoXFxwc2lfe2cnfV57LTF9KSJdLFswLDEsIlxcdHIoXFxwc2lfZ157LTF9KSIsMl0sWzEsMywiXFx0cihcXHBoaV57XzF9XFxvZG90IFcnKFxceGkpXFxvZG90IFxccGhpKSIsMl1d
\[\begin{tikzcd}
	{\sh{{\sigma}(g)}} && {\sh{{\sigma}(g')}} \\
	{\sh{{\sigma}'(g)}} && {\sh{{\sigma}'(g')}}
	\arrow["{\tr(\phi^{-1}\odot {\sigma}(\xi)\odot \phi)}", from=1-1, to=1-3]
	\arrow["{\tr(\psi_g^{-1})}"', from=1-1, to=2-1]
	\arrow["{\tr(\psi_{g'}^{-1})}", from=1-3, to=2-3]
	\arrow["{\tr(\phi^{_1}\odot {\sigma}'(\xi)\odot \phi)}"', from=2-1, to=2-3]
\end{tikzcd}\]
    is essentially the same as the diagram in \cref{fig:0:cell:functor} using \cref{ex:under:2:it:1:fig:compose}.

    For the second equality 
% https://q.uiver.app/#q=WzAsMyxbMCwwLCJcXHNoe3tcXHJob30oZil9Il0sWzEsMCwiXFxzaHt7XFxzaWdtYX0oXFxpb3RhKGYpKX0iXSxbMSwxLCJcXHNoe3tcXHNpZ21hfScoXFxpb3RhKGYpKX0iXSxbMCwxLCJcXHRyKFxcemV0YV9mXnstMX0pIl0sWzAsMiwiXFx0cihcXHpldGFfZideey0xfSkiLDJdLFsxLDIsIlxcdHIoXFxwc2lfe1xcaW90YShmKX1eey0xfSkiXV0=
\[\begin{tikzcd}
	{\sh{{\rho}(f)}} & {\sh{{\sigma}(\iota(f))}} \\
	& {\sh{{\sigma}'(\iota(f))}}
	\arrow["{\tr(\zeta_f^{-1})}", from=1-1, to=1-2]
	\arrow["{\tr(\zeta_f'^{-1})}"', from=1-1, to=2-2]
	\arrow["{\tr(\psi_{\iota(f)}^{-1})}", from=1-2, to=2-2]
\end{tikzcd}\]
we use \cite[7.5]{Ponto_2012} to compare maps without traces.   This is the diagram in \cref{fig:1:cell:functor}.
\begin{figure}
% https://q.uiver.app/#q=WzAsNSxbMCwwLCJ7XFxyaG99KGYpXFxvZG90IFxcemV0YV97Y31cXG9kb3QgXFxwc2lfe1xcaW90YShjKX0iXSxbMSwwLCJ7XFxyaG99KGYpXFxvZG90IFxcemV0YSdfe2N9Il0sWzAsMSwiXFx6ZXRhX2NcXG9kb3Qge1xcc2lnbWEgfShcXGlvdGEoZikpXFxvZG90IFxccHNpX3tcXGlvdGEoYyl9Il0sWzAsMiwiXFx6ZXRhX2NcXG9kb3QgXFxwc2lfe1xcaW90YShjKX1cXG9kb3Qge1xcc2lnbWF9JyhcXGlvdGEoZikpIl0sWzEsMiwiXFx6ZXRhJ19jXFxvZG90IHtcXHNpZ21hfScoXFxpb3RhKGYpKSJdLFswLDEsIlxcaWRcXG9kb3QgXFxiZXRhX3tjfSIsMCx7ImxldmVsIjoyfV0sWzAsMiwiXFx6ZXRhX2Zeey0xfVxcb2RvdCBcXGlkIiwwLHsibGV2ZWwiOjJ9XSxbMSw0LCIoXFx6ZXRhX2YnKSBeey0xfSIsMCx7ImxldmVsIjoyfV0sWzIsMywiXFxpZFxcb2RvdCBcXHBzaV97XFxpb3RhKGYpfV57LTF9IiwwLHsibGV2ZWwiOjJ9XSxbMyw0LCJcXGJldGFfY1xcb2RvdCBcXGlkIiwwLHsibGV2ZWwiOjJ9XV0=
\[\begin{tikzcd}
	{{\rho}(f)\odot \zeta_{c}\odot \psi_{\iota(c)}} & {{\rho}(f)\odot \zeta'_{c}} \\
	{\zeta_c\odot {\sigma }(\iota(f))\odot \psi_{\iota(c)}} \\
	{\zeta_c\odot \psi_{\iota(c)}\odot {\sigma}'(\iota(f))} & {\zeta'_c\odot {\sigma}'(\iota(f))}
	\arrow["{\id\odot \beta_{c}}", Rightarrow,  from=1-1, to=1-2]
	\arrow["{\zeta_f^{-1}\odot \id}", Rightarrow,  from=1-1, to=2-1]
	\arrow["{(\zeta_f') ^{-1}}", Rightarrow,  from=1-2, to=3-2]
	\arrow["{\id\odot \psi_{\iota(f)}^{-1}}", Rightarrow,  from=2-1, to=3-1]
	\arrow["{\beta_c\odot \id}", Rightarrow,  from=3-1, to=3-2]
\end{tikzcd}\]
\caption{The second equality in the proof of \cref{lem:1:cell:functor}}\label{fig:1:cell:functor}
\end{figure}
It commutes by assumption \cref{ex:under:2:it:1:fig:modification}.
\end{proof}

It feels a bit overkill to also state the case for 2-cells as a lemma, but it since it will make it easier to reference we follow the established pattern.

\begin{lem}\label{lem:2:cell:functor} Let $\gamma$ be a 2-cell $(\psi,\alpha)\to (\psi',\alpha')$ as in 
 \cref{ex:under:2}(\ref{ex:under:2:it:3}). If 
 $\gamma_d$ is an isomorphism for all 0-cells $d$ in $\mathcal{D}$, then the images of $(\psi,\alpha)$ and $(\psi',\alpha')$ in $((\chi_*\circ (\Lambda {\rhov}))\downarrow (\Lambda \iota^*))$ are equal.
\end{lem}

Then \cref{lem:0:cell:functor,lem:1:cell:functor,lem:2:cell:functor} are significant contributions to the proof of the following result. \begin{thm}\label{thm:trace_induced}
    If $\mathcal{B}$ has a shadow, the assignment of a 0-cell to its character extends to a strict 2-functor
    \[X\colon ({\rhov}\downarrow \iota^*)^d\to (\Lambda {\rhov}\downarrow \Lambda \iota^*).\]
\end{thm}

\begin{proof}\cref{lem:0:cell:functor,lem:1:cell:functor,lem:2:cell:functor} define the functor on 0-, 1-, and 2-cells.  
 
 Given 1-cells $({\sigma}, \zeta) \to ({\sigma}', \zeta') \to ({\sigma}'', \zeta'') $ where 
 \begin{itemize}
 \item the first is given by 1-cells $\psi_d$ and 2-cells $\psi_g, \alpha_c$ and 
 \item the second is 
given by 1-cells $\psi_d'$ and 2-cells $\psi_g', \alpha_c'$, 
\end{itemize}
then the composite is given by the 1-cells 
\[\psi_d\odot \psi_d'\] 
for $d$ in $\mathcal{D}$, and 2-cells 
\begin{equation}\label{eq:characer_as_functor_induced_1}
\psi_d\odot \psi_d'\odot {\sigma}''(d)\xRightarrow{\id\odot \psi'_g}\psi_d\odot {\sigma}'(d)\odot \psi_{d'}''\xRightarrow{\psi_g\odot \id} {\sigma}(d)\odot \psi_{d'}'\odot \psi_{d'}''
\end{equation}
and 
\[\zeta_c\odot \psi_{\iota(c)}\odot \psi_{\iota(c)}' \xRightarrow{\alpha_c\odot \id} \zeta_c'\odot \psi_{\iota(c)}' \xRightarrow{\alpha_c'}\zeta''_c\]
The image of this composite 1-cell under the definition in \cref{lem:1:cell:functor}, is the trace of \eqref{eq:characer_as_functor_induced_1}.  This   is the composites of the traces of $\psi_g$ and $\psi_g'$ and the functor  respects composition strictly.

The unit 1-cell associated to a 0-cell $({\sigma}, \zeta)$ is the 1-cells $U_{{\sigma}(d)}$ and the 2-cells 
\begin{equation}\label{eq:characer_as_functor_induced_2}
U_{{\sigma}(d)}\odot {\sigma}(g)\Rightarrow {\sigma}(g)\Rightarrow {\sigma}(g)\odot U_{{\sigma}(d)}
\end{equation}
and 
\[\zeta_c\odot U_{{\sigma}(\iota(c))}\Rightarrow \zeta_c.\]
The image of this 1-cell under the definition in \cref{lem:1:cell:functor} is 
the trace of \eqref{eq:characer_as_functor_induced_2}.
By \cite[7.4]{Ponto_2012}, this is \[\sh{\id}\colon \sh{{\sigma}(g)}\to \sh{{\sigma}(g)}\]
and so identity is strictly preserved.
\end{proof}

\begin{proof}[Proof of \cref{intro:induced}]
Recall that $I_\mathcal{C}^\mathcal{D}({\rhov})$  is a 0-cell in $({\rhov}\downarrow \iota^*)^d$.     Then by \cref{thm:trace_induced}, $\mathcal{X}_{I_\mathcal{C}^\mathcal{D}({\rhov})}$ is a 0-cell in  $\Lambda {\rhov}\downarrow \Lambda \iota^*$.  The unique map \[I_{\Lambda \mathcal{C}}^{\Lambda \mathcal{D}} (\mathcal{X}_{\rhov})\to \mathcal{X}_{I_\mathcal{C}^\mathcal{D}({\rhov})}\] exists since $I_{\Lambda \mathcal{C}}^{\Lambda \mathcal{D}} (\mathcal{X}_{\rhov})$ is the initial object in $\Lambda {\rhov}\downarrow \Lambda \iota^*$.
\end{proof}

\bibliographystyle{amsalpha2} \bibliography{refs} 

\providecommand{\bysame}{\leavevmode\hbox to3em{\hrulefill}\thinspace}
\providecommand{\MR}{\relax\ifhmode\unskip\space\fi MR }
\providecommand{\MRhref}[2]{%
  \href{http://www.ams.org/mathscinet-getitem?mr=#1}{#2}
}
\providecommand{\doi}[1]{%
  doi:\href{https://dx.doi.org/#1}{#1}}
\providecommand{\arxiv}[1]{%
  arXiv:\href{https://arxiv.org/abs/#1}{#1}}
\providecommand{\href}[2]{#2}
\begin{thebibliography}{BCSW83}

\bibitem[Bar24]{barhite2023bicategorical}
J.~Barhite, \emph{Bicategorical traces and cotraces}, Theory and Applications
  of Categories \textbf{41} (2024), no.~22, 707--759. \arxiv{2307.13070}

\bibitem[Bar11]{Bartlett}
B.~Bartlett, \emph{The geometry of unitary 2-representations of finite groups
  and their 2-characters}, Appl. Categ. Structures \textbf{19} (2011), no.~1,
  175--232. \doi{10.1007/s10485-009-9189-0} \arxiv{0807.1329}

\bibitem[BZN]{BZN}
D.~Ben-Zvi and D.~Nadler, \emph{Secondary traces}. \arxiv{1305.7177}

\bibitem[BCSW83]{betti}
R.~Betti, A.~Carboni, R.~Street, and R.~Walters, \emph{Variation through
  enrichment}, J. Pure Appl. Algebra \textbf{29} (1983), no.~2, 109--127.
  \doi{10.1016/0022-4049(83)90100-7}

\bibitem[Bor94]{Borceux}
F.~Borceux, \emph{Handbook of categorical algebra. 1}, Encyclopedia of
  Mathematics and its Applications, vol.~50, Cambridge University Press,
  Cambridge, 1994, Basic category theory.

\bibitem[CP22]{cp:iterated}
J.~A. Campbell and K.~Ponto, \emph{Iterated traces in 2-categories and
  {L}efschetz theorems}, Algebr. Geom. Topol. \textbf{22} (2022), no.~2,
  815--879. \doi{10.2140/agt.2022.22.815} \arxiv{1908.07497}

\bibitem[DP78]{Dold1978-ol}
A.~Dold and D.~Puppe, \emph{Duality, trace, and transfer}, Proceedings of the
  International Conference on Geometric Topology, Warsaw; Warsaw, 1978,
  pp.~81--102.

\bibitem[Elg07]{elgueta}
J.~Elgueta, \emph{Representation theory of 2-groups on {K}apranov and
  {V}oevodsky's 2-vector spaces}, Advances in Mathematics \textbf{213} (2007),
  no.~1, 53--92. \doi{https://doi.org/10.1016/j.aim.2006.11.010}

\bibitem[GK08]{GK}
N.~Ganter and M.~Kapranov, \emph{Representation and character theory in
  2-categories}, Adv. Math. \textbf{217} (2008), no.~5, 2268--2300.
  \doi{10.1016/j.aim.2007.10.004} \arxiv{math/0602510}

\bibitem[Gra74]{gray}
J.~W. Gray, \emph{Formal category theory: adjointness for {$2$}-categories},
  Lecture Notes in Mathematics, Vol. 391, Springer-Verlag, Berlin-New York,
  1974.

\bibitem[HKR00]{HKR}
M.~J. Hopkins, N.~J. Kuhn, and D.~C. Ravenel, \emph{Generalized group
  characters and complex oriented cohomology theories}, J. Amer. Math. Soc.
  \textbf{13} (2000), no.~3, 553--594. \doi{10.1090/S0894-0347-00-00332-5}

\bibitem[Lei]{leinster}
T.~Leinster, \emph{Basic bicategories}. \arxiv{math/9810017}

\bibitem[MP85]{bicat_coherence}
S.~Maclane and R.~Par\'{e}, \emph{Coherence for bicategories and indexed
  categories}, Journal of Pure and Applied Algebra \textbf{37} (1985), 59--80.

\bibitem[MP22]{MP:coherence}
C.~Malkiewich and K.~Ponto, \emph{Coherence for bicategories, lax functors, and
  shadows}, Theory Appl. Categ. \textbf{38} (2022), Paper No. 12, 328--373.
  \doi{10.1086/288336} \arxiv{2109.01249}

\bibitem[MS06]{may2004parametrized}
J.~P. May and J.~Sigurdsson, \emph{Parametrized homotopy theory}, Mathematical
  Surveys and Monographs, vol. 132, American Mathematical Society, Providence,
  RI, 2006. \doi{10.1090/surv/132} \arxiv{math/0411656}

\bibitem[Pon10]{p:thesis}
K.~Ponto, \emph{Fixed point theory and trace for bicategories}, Ast\'{e}risque
  (2010), no.~333, xii+102. \arxiv{0807.1471}

\bibitem[PS12]{Ponto_2012}
K.~Ponto and M.~Shulman, \emph{Shadows and traces in bicategories}, Journal of
  Homotopy and Related Structures \textbf{8} (2012), no.~2, 151--200.
  \doi{10.1007/s40062-012-0017-0} \arxiv{0910.1306}

\bibitem[Sta65]{stallings}
J.~Stallings, \emph{Centerless groups---an algebraic formulation of
  {G}ottlieb's theorem}, Topology \textbf{4} (1965), 129--134.
  \doi{10.1016/0040-9383(65)90060-1}

\end{thebibliography}

\end{document}